\documentclass[11pt, a4paper]{article}
\usepackage[utf8]{inputenc}
\usepackage{mathtools}
\usepackage{amsmath}
\usepackage{amsthm}
\usepackage{algorithm2e}
\usepackage{algorithmic}
\RestyleAlgo{ruled}
\SetKwInput{KwInput}{Input}
\SetKwInput{KwOutput}{Output}
\usepackage{algorithmic}
\usepackage{amssymb}
\usepackage{parskip}
\usepackage{epsfig}
\usepackage[colorlinks, citecolor=hanblue,linkcolor=mordantred19,urlcolor=darkgreen]{hyperref}

\usepackage{stackengine}
\mathtoolsset{showonlyrefs}
\usepackage{xcolor}
\usepackage[a4paper, margin=2.5cm]{geometry}
\usepackage{comment}
\usepackage{bbm}
\usepackage{mathrsfs}
\usepackage{subfigure}
\usepackage{float}
\usepackage{tikz, pgfplots}
\usetikzlibrary{decorations.pathreplacing,calligraphy}
\pgfplotsset{compat=1.18}
\usepackage[export]{adjustbox}
\DeclareMathOperator*{\argmax}{arg\,max}
\DeclareMathOperator*{\argmin}{arg\,min}
\usepackage{multirow}
\usepackage{setspace}

\usepackage{color}
\definecolor{hanblue}{rgb}{0.27, 0.42, 0.81}
\definecolor{mordantred19}{rgb}{0.68, 0.05, 0.0}
\definecolor{darkgreen}{rgb}{0.0, 0.38, 0.12}
\definecolor{red}{rgb}{0.8, 0.0, 0.0}
\definecolor{green}{rgb}{0.0, 0.5, 0.0}

\newcommand{\eps}{\varepsilon}

\DeclareMathOperator{\supp}{supp}

\usepackage{mathtools}
\usepackage{hyperref}
\mathtoolsset{showonlyrefs}

\renewcommand{\div}{{\rm div}\,}
\newcommand{\R}{\mathbb{R}}
\newcommand{\1}{\mathbbm{1}}

\newcommand{\N}{\mathbb{N}}

\newcommand{\cC}{\mathcal{C}}
\renewcommand{\S}{\mathcal{S}}
\newcommand{\T}{\mathcal{T}}
\newcommand{\A}{\mathcal{A}}

\newcommand{\BV}{\operatorname{BV}}

\newcommand{\Lq}{L^q(\Omega)}

\newcommand{\TV}{\operatorname{TV}}
\newcommand{\Per}{\operatorname{Per}}
\newcommand{\dist}{\operatorname{dist}}
\newcommand{\cl}{\operatorname{cl}}

\newcommand{\wkto}{\rightharpoonup}
\newcommand{\wksto}{\stackrel{\ast}{\rightharpoonup}}

\newcommand{\Span}{\operatorname{span}}
\newcommand{\dd}{\, \mathrm{d}}

\newcommand{\bH}{H^1(\partial\bar{E}^{\gamma})}

\newcommand\restr[2]{{% we make the whole thing an ordinary symbol
 \left.\kern-\nulldelimiterspace % automatically resize the bar with \right
 #1 % the function
 \vphantom{\big|} % pretend it's a little taller at normal size
 \right|_{#2} % this is the delimiter
 }}

\usepackage{enumitem}
\numberwithin{equation}{section}

\theoremstyle{plain}

\newtheorem{theorem}{Theorem}[section]

\newtheorem{lemma}[theorem]{Lemma}
\newtheorem{proposition}[theorem]{Proposition}
\newtheorem{corollary}[theorem]{Corollary}

\newtheorem{assumption}{Assumption}

\newtheorem{remark}[theorem]{Remark}

\theoremstyle{definition}

\definecolor{hanblue}{rgb}{0.27, 0.42, 0.81}

\usepackage{scalerel}

\allowdisplaybreaks[1]

\title{Linear convergence of a one-cut conditional gradient method for total variation regularization}
\author{Giacomo Cristinelli$^\ast$, Jos\'e A. Iglesias\thanks{Department of Applied Mathematics, University of Twente, 7500AE Enschede, The Netherlands \newline (\texttt{g.cristinelli@utwente.nl, jose.iglesias{@}utwente.nl})} , Daniel Walter\thanks{Institut f\"ur Mathematik, Humboldt-Universit\"at zu Berlin, 10117 Berlin, Germany \newline (\texttt{daniel.walter@hu-berlin.de})}}
\date{}

\begin{document} 
\maketitle

\begin{abstract}
We introduce a fully-corrective generalized conditional gradient method for convex minimization problems involving total variation regularization on multidimensional domains. It relies on alternatively updating an active set of subsets of the spatial domain and an iterate given by a conic combination of the associated characteristic functions. Different to previous approaches in the same spirit, the computation of a new candidate set only requires the solution of one prescribed mean curvature problem, instead of the resolution of a fractional minimization task analogous to finding a generalized Cheeger set. After discretization, the former can be realized by a single run of a graph cut algorithm leading to significant speedup in practice. We prove the global sublinear convergence of the resulting method, under mild assumptions, and its asymptotic linear convergence in a more restrictive two-dimensional setting which uses results of stability of surfaces of prescribed mean curvature under perturbations of the curvature. Finally, we numerically demonstrate this convergence behavior in some model PDE-constrained minimization problems.
\end{abstract}
 

\section{Introduction}
In this paper, we consider the following minimization problem 
\begin{align} \label{def:BVprob}
 \min_{0 \leq u \in \Lq} J(u) \coloneqq \left\lbrack F(Ku)+\TV(u,\Omega) \right\rbrack \tag{$\cal{P}$}
\end{align}
where $\Omega \subset \R^d $ is a bounded domain with strongly Lipschitz boundary (in the sense that, up to a rotation, it can be locally written as the subgraph of a Lipschitz function),~$q=d/(d-1)$ and~$d \geq 1$, $K$ is a bounded linear operator mapping to some separable Hilbert space $Y$, and 
\begin{equation}\label{eq:defTV}\TV(u,\Omega):=\sup\left\{\int_\Omega u\,\div\psi \dd x \,\middle\vert\, \psi\in \cC^1_c(\Omega; \R^d), \|\psi\|_{\cC(\Omega;\R^d)}\leq 1\right\}\end{equation}
is the isotropic total variation of $u$ in $\Omega$. Incorporating the latter as a regularizer in inverse problems and optimal control tasks formalizes the modelling assumption that the sought-for solutions should be piecewise constant. From a geometrical perspective, this is motivated by the characterization (see e.g.~\cite[Prop.~8]{AmbCasMasMor01}) of the extreme points of the total variation ball $\{u \,\vert\, \TV(u,\R^d) \leq 1\}$ as characteristic functions of simple sets (roughly speaking, simply connected). In the case where $Y$ is finitely dimensional, this can be rigorously justified by the use of convex representer theorems \cite{BoyEtAl19, BreCar20}.

More recently, numerical algorithms exploiting this expected sparsity structure were introduced by \cite{CriIglWal23, DecDuvPet23} based on accelerated variants of generalized conditional gradient methods. More in detail, and neglecting the inequality constraints for now, these produce a sequence of piecewise constant iterates $u_k=\sum_j \lambda^j_k \1_{E^j_k}$ by alternating between a set insertion subproblem, whose solution is subsequently added to the iterate, and a finite-dimensional but convex coefficient update problem in order to adjust the weights appearing in the linear combination. 
The method in this paper is an improvement of the one proposed in \cite{CriIglWal23}, so some comparison between them and the method in \cite{DecDuvPet23} is warranted. 
In our view, the main difference is the setting in which each of them is applicable. The present paper and \cite{CriIglWal23} focus on situations where $K$ is the solution operator associated with a linear PDE, and whose corresponding kernel is generally unavailable. By comparison, \cite{DecDuvPet23} considers integral operators whose kernels are given explicitly. 
This difference in perspective carries over to the representation of the solutions of the set insertion problem. 
In the present paper and \cite{CriIglWal23}, candidate sets are represented as unions of simplices in a triangulation compatible with the finite-element discretization of the underlying PDE.
In \cite{DecDuvPet23}, the authors exploit an off-grid polygonal description of the sets to implement additional nonconvex deformation steps in which the global energy is used to modify the sets $E^j_k$. Although such steps would be desirable in the setting of \cite{CriIglWal23} and the present paper, their implementation would require global mesh deformations and more involved shape derivatives, and is therefore left for future work.

The main difference between \cite{CriIglWal23} and the present work is the treatment of the set insertion subproblem that needs to be solved at each iteration. The method in \cite{CriIglWal23} uses a Dinkelbach-Newton procedure that requires solving a sequence of prescribed mean curvature problems which, after discretization, can be rewritten as minimum cut problems on the dual graph of the mesh. In the present paper, we instead exploit the fact that the set of minimizers to \eqref{def:BVprob} is, under mild assumptions, bounded in $L^\infty$, see \cite{BreIglMer22}. This leads to a simpler set insertion subproblem that requires solving a single shape optimization problem, or equivalently a single prescribed mean curvature problem, per iteration. Moreover, we split the resulting set into its indecomposable components and add all of the corresponding characteristic functions, which provides greater flexibility at each iteration.

The main contributions of the present paper are twofold: first, again relying on the interpretation as an accelerated conditional gradient method, we derive global convergence of the resulting method together with sublinear rates for the objective functional values. Second, going beyond standard techniques, we are able to prove an asymptotic linear rate of convergence, matching numerical observations, provided that the optimal solution is piecewise constant and supported on a finite number of well separated sets, and the dual variable satisfies certain growth assumptions in terms of boundary deformations of those sets. These conditions are strongly inspired by the framework proposed in \cite{DecDuvPet24} for the analysis of total variation regularized inverse problems in the low noise regime. This second type of convergence rate result was missing even for the previous related methods, which by their direct use of extreme points are closer to the available literature on linear convergence guarantees for generalized conditional gradient methods, and in particular to \cite{BreCarFanWal24} whose methods we build upon.

\subsection{Notation and standing assumptions}
The definition of the total variation in \eqref{eq:defTV} implies that if $\TV(u,\Omega)<\infty$ the distributional derivative $Du$ is an $\R^d$-valued Radon measure, and the space of $L^1(\Omega)$ functions such that this holds is denoted as $\BV(\Omega)$. The assumed regularity of the boundary of $\Omega$ implies \cite[Prop.~3.21]{AmbFusPal00} that it is an extension domain, namely, that the extension by zero of any function in $\BV(\Omega)$ belongs to $\BV(\R^d)$.

Through the total variation defined in \eqref{eq:defTV} we can also define the perimeter in $\Omega$ of a subset $E \subset \Omega$ as $\Per(E,\Omega):=\TV(\1_E,\Omega)$ where $\1_E$ is the characteristic function of $E$, and say that $E$ is of finite perimeter in $\Omega$ when this quantity is finite. The perimeter is invariant with respect to modifications of $E$ of Lebesgue measure zero, but we can take a representative of $E$ for which the topological boundary equals the support of the derivative of $\1_E$, which satisfies \cite[Prop.~12.19]{Mag12}
\begin{equation}\label{eq:bdy}\partial E = \supp D\1_E =\left\{ x \in \R^d \,\middle\vert\, 0< \frac{|E \cap B(x,r)|}{|B(x,r)|} < 1 \text{ for all } r>0\right\}.\end{equation}
In the rest of the paper, unless otherwise specified, we use this representative for sets of finite perimeter. Following \cite{AmbCasMasMor01}, we say that such a set $E$ is \textit{decomposable} if it admits a partition $E=E_1 \cup E_2$ with $E_1 \cap E_2 = \emptyset$, $E_1,E_2$ have positive Lebesgue measure, and $\Per(E,\Omega)=\Per(E_1,\Omega)+\Per(E_2,\Omega)$, and \textit{indecomposable} if no such decomposition exists. By $\Omega$ being an extension domain, we can consider any set $E \subset \Omega$ with $\Per(E,\Omega)<\infty$ as a finite perimeter set in $\R^d$ to apply \cite[Thm.~1]{AmbCasMasMor01} to obtain a decomposition of it into countably many indecomposable components. 

A thorough treatment of the total variation, associated function spaces, and minimization problems involving perimeters can be found in the monographs \cite{AmbFusPal00} and \cite{Mag12}. We will also make use of the distance of a point $x \in \R^d$ to a set $A \subset \R^d$ defined by $\dist(x,A) = \inf_{y \in A}|x-y|$, and the distance between two sets defined as $\dist(A,B) := \inf_{x\in A,y \in B}|x-y|$ for $A,B \subset \R^d$. The closure of such a set is denoted by $\cl A$. Finally, we denote the symmetric difference of two sets as $A \Delta B = A \setminus B \cup B \setminus A$.

We often consider functions $u \in \cC^m(B)$ for not necessarily open sets $B \subset \R^d$, which means that these can be extended to a function $v \in \cC^m(\R^d)$ which coincides with $u$ on $B$. This space can be endowed with the norm
\[\|u\|_{\cC^m(B)} := \inf \left\{\|v\|_{\cC^m(\R^d)} \, \middle \vert \, v \in \cC^m(\R^d)\text{ and }v \equiv u\text{ on }B\right\},\]
where in turn for $v \in \cC^m(O)$ with $O \subseteq \R^d$ open,
\[\|v\|_{\cC^m(O)} := \max_{|\beta|\leq m} \sup_{x \in O} \big|\partial^\beta v(x)\big|.\]
We say that an open set $G \subset \R^d$ has $\cC^m$ boundary if it is locally the subgraph of a $\cC^m$ function, meaning that for any $x \in \partial G$, there is $r>0$ such that $G \cap B(x,r)$ is the subgraph of a $\cC^m$ function defined on some affine hyperplane of $\R^d$, see \cite[Definition 5.2]{DelZol11}. For such a set $G$, there is a constant $c >0$ such that for any function $\varphi \in \cC^{m-1}(\partial G)$ with $\|\varphi\|_{\cC^{m-1}}(\partial G)\leq c$, the map $\mathrm{Id}+\varphi\, n_G:\partial G \to \R^d$ where $n_{G}$ is the outer normal vector can be extended to a $\cC^{m-1}$ function with continuous inverse $\psi:\R^d \to \R^d$, see \cite[Ch.~4, Sec.~3.3.2]{DelZol11}. We then define the deformation of a connected component $G^j$ of $G$ as 
\[(\varphi)_{\#}(G^j):= \psi(G^j),\]
which by virtue of $\varphi$ being defined on $\partial G$, does not depend on the choice of such $\psi$.

For an ordered finite set of sets of finite perimeter $\A=\{E^j\}^N_{j=1}$, and with a small abuse of notation, we further introduce
\begin{align}
 \mathcal{U}_{\A}(\lambda) \coloneqq \sum_{E^j \in \A} \lambda^j \1_{E^j}, \quad r_{\A}(\lambda) \coloneqq F(K\mathcal{U}_{\A}(\lambda))+ \sum_{E^j \in \A} \lambda^j \Per(E^j,\Omega) - \min \eqref{def:BVprob}.
\end{align}

Regarding the fidelity term $F$ as well as the forward operator $K$ in \eqref{def:BVprob}, the following standing assumptions are made.
\begin{assumption}\label{ass:functions}
For a separable Hilbert space $Y$ with inner product $(\cdot, \cdot)_Y$ and induced norm $\|\cdot\|_Y$, assume that:
\begin{enumerate}[label=\textbf{A\arabic*}]
    \item \label{item:assonK} The operator $K \colon L^q(\Omega) \to Y$ is sequentially weak-to-strong continuous, i.e. for $u \in L^q(\Omega)$ we have that $u_n \rightharpoonup u$ implies $Ku_n \rightarrow Ku.$
    \item \label{item:assonF} The mapping $F \colon Y \to \R_+ $ is strictly convex and continuously Fr\'echet differentiable. Its gradient $\nabla F \colon Y \to Y $ is Lipschitz-continuous, i.e. there is $L_{\nabla F} >0$ such that
    \begin{equation}
        \|\nabla F(y_1)-\nabla F(y_2)\|_Y \leq L_{\nabla F} \|y_1-y_2\|_{Y} \quad \text{for all} \quad y_1,y_2 \in Y.
    \end{equation}
    \item\label{item:assonsublevel} The sublevel sets $E_J(u) \coloneqq \{\,v \in L^q(\Omega)\;|\; J(v) \leq J(u) \,\}$ are bounded for every $u \in \BV(\Omega)$.
\end{enumerate}
\end{assumption}
Note that Assumption \eqref{item:assonsublevel} together with $\Per(\1_\Omega,\Omega)=0$ implies $K \1_{\Omega} \neq 0$. Vice versa, given the latter, we can formulate sufficient conditions on $F$ such that Assumption \eqref{item:assonsublevel} holds, see, e.g., \cite{CriIglWal23}. Since, in the following, we will only rely on the boundedness of the sublevel sets of $J$, we prefer to work with \eqref{item:assonsublevel} instead of more specific conditions.

Finally, we emphasize that quadratic fidelity terms $F(y)=(1/2) \|\cdot-y_d\|^2_Y$ satisfy Assumption \ref{ass:functions} if $K \1_{\Omega} \neq 0$.

The remainder of the paper is structured as follows: In Section \ref{sec:start}, we collect relevant results regarding existence and properties of minimizers to \eqref{def:BVprob} and prescribed curvature problems. Section \ref{section:onecut} introduces the new algorithm, proving its global convergence as well the asymptotic, improved convergence behavior. The paper is concluded by applying the presented method to PDE-constrained minimization problems in Section \ref{sec:numerics}.
\section{Existence of minimizers and optimality conditions} \label{sec:start} 
From Assumption \ref{ass:functions}, we conclude that the sublevel sets of $J$ are weakly compact. Hence, existence of minimizers to \eqref{def:BVprob} follows by standard arguments. We skip the proof of existence for the sake of brevity, with the bound below following directly from Assumption \eqref{item:assonsublevel}.
\begin{theorem} \label{thm:existence}
Problem \eqref{def:BVprob} admits at least one minimizer and we have $K \bar{u}_1=K \bar{u}_2$ for all solutions $\bar{u}_1, \bar{u}_2$ to \eqref{def:BVprob}. Moreover, there is $M_q >0$ such that $\|\bar{u} \|_{L^q(\Omega)} \leq M_q$ for any solution $\bar{u}$ of \eqref{def:BVprob}.
\end{theorem}
Given $p \in L^d(\Omega)$, we will heavily rely on properties of minimizers to the associated prescribed mean curvature problem 
\begin{equation} \label{def:prescribedcurvature}
    \min_{E \subset \Omega} \left\lbrack -\int_E p \dd x + \Per(E,\Omega) \right \rbrack,\tag{$\cal{MC}$}
\end{equation}
for which we collect known results on existence and regularity of minimizers in Proposition \ref{prop:existenceMC}. The significance of this problem is also foreshadowed by the following first-order optimality condition (\cite[Prop.~3]{ChaDuvPeyPoo17}, \cite[Lem.~1]{BreIglMer22}):
\begin{proposition} \label{prop:firstorderopt}
Given $\bar{u} \in \BV(\Omega)$, $\bar{u} \geq 0$, as well as $\bar{p}\coloneqq-K^* \nabla F(K \bar{u}) \in L^d(\Omega)$. Then $\bar{u}$ is a minimizer of \eqref{def:BVprob} if and only if one of the following three, equivalent, properties hold: 
\begin{itemize}
    \item We have $\int_{\Omega}\bar{p}u \dd x \leq \TV(u,\Omega)$ for all $u \in L^q(\Omega)$ with $u \geq 0$, and $\int_\Omega \bar{p} \bar{u} \dd x= \TV(\bar{u},\Omega)$.
    \item We have $\int_{\Omega} \bar{p}u \dd x \leq \TV(u,\Omega)$ for all $u \in L^q(\Omega)$ with $u \geq 0$, and for a.e. $s\geq 0$ the level set $E^s \coloneqq \{\,x \in \Omega\;|\;u(x)>s\,\}$ satisfies $\int_{E^s} \bar{p} \dd x= \Per(E^s,\Omega)$.
    \item For a.e. $s\geq 0$, the level set $E^s$ is a minimizer of \eqref{def:prescribedcurvature} for $p=\bar{p}$.
\end{itemize}
\end{proposition}
We emphasize that the optimal dual variable $\bar{p}$ is fully characterized by the optimal observation $ \bar{y}=K \bar{u}$ which is the same for every minimizer to \eqref{def:BVprob}, see Theorem \ref{thm:existence}. 

Note that solutions to \eqref{def:prescribedcurvature} are far from being unique, in fact the previous proposition tells us that \textit{all} level sets of a solution of \eqref{def:BVprob} are solutions to \eqref{def:prescribedcurvature} for the same $p=\bar{p}$.
\begin{lemma}\label{lem:maxmin}
Unions and intersections of minimizers of \eqref{def:prescribedcurvature} are still minimizers, and there is a unique maximal one with respect to inclusion.
\end{lemma}
\begin{proof}
The intersection and union property is a direct consequence of submodularity of the perimeter, see for example \cite[Prop~3.3]{IglMer21} or \cite[Prop.~3.3]{DecDuvPet24}. The existence of the maximal minimizer then follows directly.
\end{proof}
In several points of the analysis, it will be important for us to know that the indecomposable components of minimizers of \eqref{def:prescribedcurvature} along a sequence of curvatures do not degenerate either in mass or in perimeter:
\begin{lemma}\label{lem:upperlowerbounds}
Let $p_n$ be a strongly converging sequence in $L^d(\Omega)$. Then there are constants $c,C>0$ for which
\begin{equation}
c \leq \big|\bar{E}^m_n\big| \leq C, \quad \Per\big(\bar{E}^m_n, \Omega\big) \leq C,
\end{equation}
where $\bar{E}^m_n$ is any indecomposable component of any nontrivial minimizer $\bar{E}_n$ of \eqref{def:prescribedcurvature} with $p_n$ satisfying $|\bar{E}_n| \in (0,|\Omega|)$. If additionally there is $n_0$ for which $\Per\big(\bar{E}_n, \Omega\big) = \Per\big(\bar{E}_n, \R^d\big)$ for all $n \geq n_0$, then also
\begin{equation}
\Per\big(\bar{E}^m_n, \Omega\big) \geq d \big(|B(0,1)|c\big)^{\frac{1}{d}} >0 \quad \text{for all }n \geq n_0.
\end{equation}
\end{lemma}
\begin{proof}
Notice that by optimality and the fact that $\bar{E}_n$ decomposes into the $\bar{E}^m_n$,
\[\Per(\bar{E}^m_n, \Omega) \leq \int_{\bar{E}^m_n} p_n \dd x \quad \text{for all }m,\]
since if this would not hold for some $\bar{E}^{j_0}_n$, then $\bar{E}_n \setminus \bar{E}^{j_0}_n$ would have a lower cost than $\bar{E}_n$ (that is, indecomposable components of minimizers of \eqref{def:prescribedcurvature} are also minimizers, as observed for example in \cite[Rem.~4]{ChaDuvPeyPoo17}). The previous inequality implies
\begin{equation}\label{eq:perholder}
\Per(\bar{E}^m_n, \Omega) \leq \big|\bar{E}^m_n\big|^{\frac{d-1}{d}} \|p_n\|_{L^d(\bar{E}^m_n}),
\end{equation}
which in combination with $\bar{E}^m_n \subset \Omega$ implies the upper bounds.

For the lower bound on mass, we can assume that $0<|\bar{E}^m_n| \leq |\Omega|/2$. Using the Sobolev inequality \cite[Thm.~3.47]{AmbFusPal00} $\TV(u,\Omega)\geq C(\Omega)\|u - [u]_{\Omega}\|_{L^{d/(d-1)}(\Omega)}$ for all $u \in \BV(\Omega)$, as in \cite[Sec.~4.3]{IglMerSch18} and \cite[Sec.~6]{IglMer20} and the binomial theorem we obtain for all $E \subset \Omega$ that
\begin{equation}\begin{aligned}
\Per(\bar{E}^m_n,\Omega) &\geq C(\Omega)\Bigg(\big|\bar{E}^m_n\big|\frac{\big|\Omega \setminus \bar{E}^m_n\big|^{\frac{d}{d-1}}}{|\Omega|^{\frac{d}{d-1}}} + \big|\Omega \setminus \bar{E}^m_n\big|\frac{\big|\bar{E}^m_n\big|^{\frac{d}{d-1}}}{|\Omega|^{\frac{d}{d-1}}}\Bigg)^{\frac{d-1}{d}} \\ &= C(\Omega)\left(\frac{\big|\bar{E}^m_n\big|\big|\Omega \setminus \bar{E}^m_n\big|}{|\Omega|}\right)^{\frac{d-1}{d}} \left(\frac{\big|\bar{E}^m_n\big|^{\frac{1}{d}}+\big(\big|\Omega\big| - \big|\bar{E}^m_n\big|\big)^{\frac{1}{d}}}{|\Omega|^{\frac{1}{d}}}\right)^{\frac{d-1}{d}}  
\\ &\geq C(\Omega)\left(\frac{\big|\bar{E}^m_n\big|\big|\Omega \setminus \bar{E}^m_n\big|}{|\Omega|}\right)^{\frac{d-1}{d}},
\end{aligned}\end{equation}
which by $0<|\bar{E}^m_n| \leq |\Omega|/2$ and in combination with \eqref{eq:perholder} gives
\[\frac{C(\Omega)}{2^{(d-1)/d}} \leq \|p_n\|_{L^d(\bar{E}^m_n)}.\]
But since by their strong convergence the $p_n$ are equiintegrable in $L^d(\Omega)$, there is some $c_0 \in (0,1/2)$ such that if $|E|<c_0$ then $\|p_n\|_{L^d(E)} < C(\Omega)/2^{(d-1)/d}$ for all $n$, which immediately gives a contradiction. This proves that $|\bar{E}^m_n| \geq c_0$. Choosing $c = \min(c_0,|\Omega|/2)$ we obtain the claimed lower bound.

In case $\Per\big(\bar{E}_n, \Omega\big) = \Per\big(\bar{E}_n, \R^d\big)$ we also have $\Per\big(\bar{E}^m_n, \Omega\big) = \Per\big(\bar{E}^m_n, \R^d\big)$ for all $j$. Since $|\bar{E}^m_n\big| < \infty$, by the isoperimetric inequality in $\R^d$ \cite[Thm.~14.1]{Mag12} and the previous bound we get
\[\Per\big(\bar{E}^m_n, \Omega\big) \geq d|B(0,1)|^{\frac{1}{d}} |\bar{E}^m_n|^{\frac{1}{d}} \geq d \big(|B(0,1)|c\big)^{\frac{1}{d}}.\]
\end{proof}

The preceding bounds imply that minimizers of \eqref{def:BVprob} are essentially bounded:
\begin{lemma} \label{lem:uniformlinfty}
There is $M_\infty >0$ such that $\|\bar{u} \|_{L^\infty(\Omega)} \leq M_\infty$ for any solution $\bar{u}$ of \eqref{def:BVprob}.
\end{lemma}
\begin{proof}
Since $\bar{p} \in L^d(\Omega)$ and this dual variable is determined by $K \bar{u}$ which is the same for all possible solutions of \eqref{def:BVprob}, using Proposition \ref{prop:firstorderopt} we get that the corresponding level sets $E^s$ are minimizers of \eqref{def:prescribedcurvature} with the same $\bar{p}$. However, by the lower bounds of Lemma \ref{lem:upperlowerbounds} we also have that there is some $c>0$ such that $|E^s| > c$ for all $s$ for which $|E^s| \neq 0$, where this $c$ only depends on $\bar{p}$. Since the level sets are nested and decreasing in $s$, this means that there is some $s_0$ such that $|E^s| = 0$ for all $s \geq s_0$, since otherwise it could not be that $\bar{u} \in L^1(\Omega)$. This $s_0$ is precisely the required bound. We point to \cite{BreIglMer22} for further extensions and consequences of such computations for $L^\infty$ bounds.
\end{proof}

\section{A one-cut generalized conditional gradient method} \label{section:onecut}
Finally, we introduce an algorithm in the spirit of \cite{CriIglWal23} which alternates between updating an active set of sets $\mathcal{A}_k $ and an iterate $u_k$ which is a conic combination of the characteristic functions of the elements of the former. In contrast to this prior work, however, we only require the solution of a single prescribed curvature problem at every iteration, 
\begin{align}
  \bar{E}_k \in \argmin_{E \subset \Omega} \left \lbrack- \int_E p_k \dd x + \Per(E, \Omega)  \right \rbrack \quad \text{where} \quad p_k=-K^* \nabla F(Ku_k),
\end{align}
instead of a sequence of similar problems as described in \cite{CriIglWal23}. The resulting method is summarized in Algorithm \ref{alg:abstractonecut}. Note that, instead of adding the computed set directly to $\A_k$, we first decompose $\bar{E}_k$ into its finitely many indecomposable components and use the latter for the update of the active set. This decomposition is finite because, by $p_k \in L^d(\Omega)$ and Lemma \ref{lem:upperlowerbounds}, there is a lower bound on the volume of each component. Modifying the algorithm in this way allows for more flexible updates of the iterate since the coefficients for each component are decoupled, a refined convergence analysis and, eventually, linear convergence of the resulting method.  
\begin{algorithm}[ht]
\setstretch{1.15}
\caption{One-Cut FC-GCG for Problem~\eqref{def:BVprob}}\label{alg:abstractonecut}
 \KwInput{$u_0=0$, $\mathcal A_0=\emptyset$}
 \For{$k=0,1,2,...$}{
   $p_k\leftarrow -K^*\nabla F(K u_k)$\\
   Find $\bar{E}_k$ with
   \begin{equation}
     \bar{E}_k \in \argmin_{E \subset \Omega} \left \lbrack -\int_E p_k \dd x+ \Per(E,\Omega)\right\rbrack. 
   \end{equation}
   \If{$\int_{\bar{E}_k} p_k \dd x \leq \Per(\bar{E}_k,\Omega)$}
   {Terminate with a solution $\bar{u}=u_k$ to~\eqref{def:BVprob}}
   Decompose $\bar{E}_k= \bigcup^{n_k}_{m=1} \bar{E}^m_k$, $\bar{E}^m_k$ indecomposable. \\
   Update active set:
   \begin{equation}
     \A_{k,+} \leftarrow \left\{E^j_{k,+}\right\}^{\#\A_{k,+}}_{j=1} \coloneqq \A_k \cup \{\bar{E}^m_k\}^{n_k}_{m=1}. 
   \end{equation}\\
Update iterate:
   \begin{align}
   &\lambda_{k,+} \in \argmin_{\lambda \geq 0} \Bigg\lbrack F\left(K\mathcal{U}_{\mathcal{A}_{k,+}}(\lambda)\right)+ \sum_{E^j_{k,+} \in \A_{k,+}}\lambda^j \Per(E^j_{k,+},\Omega) \Bigg \rbrack,\\
   &u_{k+1}\leftarrow \mathcal{U}_{\mathcal{A}_{k,+}}(\lambda_{k,+}).
   \end{align}
 $\A_{k+1} \leftarrow \A_{k,+}  \setminus \left\{\,E^j_{k,+}\;\middle\vert\;\lambda^j_{k,+}=0\,\right\}$
 }
\setstretch{1}
\end{algorithm}
The remainder of this section is dedicated to its convergence analysis, starting with the derivation of a global, but slow, rate of convergence for the residual
\begin{align}
  r_J(u) \coloneqq J(u)- \min \eqref{def:BVprob}
\end{align}
in Section \ref{subsec:sublinear}, before proving an accelerated but asymptotic behavior in Section \ref{subsec:linear} under additional structural assumptions.

Throughout the following, we assume that Algorithm \ref{alg:abstractonecut} does not terminate after finitely many steps. By construction, this implies that Algorithm \ref{alg:abstractonecut} generates sequences $u_k$, $p_k$, $\A_k=\{E^j_k\}^{\# \A_k}_{j=1}$ as well as $y_k =Ku_k$ and $\lambda_k \in \R^{\# \A_k}  $, $k \in \N$ where each $E^j_k$ is indecomposable and
\begin{equation} \label{eq:formulauk}
  u_k= \mathcal{U}_{\A_k}(\lambda_k), \ \  \lambda^j_k >0, \ \  \lambda_k \in \argmin_{\lambda \geq 0} \Bigg\lbrack F\left(K\mathcal{U}_{\mathcal{A}_k}(\lambda)\right)+ \sum_{E^j_k \in \A_k}\lambda^j \Per(E^j_k,\Omega) \Bigg\rbrack.
\end{equation}
Note that the positivity of the coefficients is ensured by the pruning step in the last line of Algorithm \ref{alg:abstractonecut}. Moreover, the latter also implies
\begin{equation}
    u_{k+1}=\mathcal{U}_{\A_{k+1}}(\lambda_{k+1})=\mathcal{U}_{\A_{k,+}}(\lambda_{k,+}).
\end{equation}
\begin{remark} \label{rem:extension}
 From a practical perspective, it might be advantageous to require $\Omega \in \mathcal{A}_k$ for every $k \in \N$, i.e. the constant function $\1_\Omega$, with $\Per(\Omega, \Omega)=0$ is inserted a priori, never removed afterwards and the associated coefficient $\lambda_\Omega \geq 0$ is optimized in every iteration. Moreover, this also represents an elegant way to extend the presented method to problems without nonnegativity constraints. In fact, optimizing $\lambda_\Omega$ without constraints  we obtain $\int_\Omega p_k  \dd x=0 $ for all $k \geq 1$ by inspecting the first-order optimality conditions, see e.g. \cite{CriIglWal23}, and thus
 \begin{equation}
     \int_{E} p_k \dd x= -\int_{\Omega\setminus E} p_k \dd x,\quad \Per(E,\Omega)=\Per(\Omega\setminus E,\Omega) \quad \text{for all} \quad E \subset \Omega. 
 \end{equation}
 Consequently, characteristic functions with a negative sign are introduced implicitly by inserting the complement of the corresponding set and optimizing $\lambda_\Omega$.  
\end{remark}
\subsection{A motivation of Algorithm \ref{alg:abstractonecut}}\label{sec:derivation}
Before proceeding to the convergence analysis of Algorithm \ref{alg:abstractonecut}, we give a motivation of it in the perspective of generalized conditional gradient methods (GCG), as first introduced in \cite{MinFuk81}. As a starting point, recall that the set of minimizers to \eqref{def:BVprob} is bounded in $L^\infty$ by a constant $M_\infty$. As a consequence, modifying \eqref{def:BVprob} by adding pointwise constraints 
\begin{align}
 \min_{u \in L^q(\Omega)}\left\lbrack F(Ku)+\TV(u,\Omega) \right\rbrack \quad \text{s.t.} \quad 0 \leq u \leq M_{\infty}
\end{align}
does not change its minimizers. Given the current iterate $u_k$, we can now apply a generalized conditional gradient method to this surrogate which leads to
\begin{equation}
    u^s_k= u_k+ s(v_k-u_k) \quad \text{where} \quad  v_k \in \argmin_{0\leq u \leq M_\infty} \bigg( -\int_{\Omega} p_k u \dd x + \TV(u, \Omega) \bigg)
\end{equation}
and $s \in [0,1]$ is a suitable stepsize. By standard arguments using the coarea and layer-cake formulas (see for example \cite[Thm.~2.2]{BurDonHin12} or \cite[Prop.~3]{ChaDuvPeyPoo17} in a slightly different situation), we now can choose
\begin{equation}
    v_k= M_\infty \1_{\bar{E}_k} \quad \text{where} \quad  \bar{E}_k \in  \argmin_{E \subset \Omega} \left\lbrack -\int_{E} p_k \dd x+ \Per(E,\Omega) \right \rbrack 
\end{equation}
motivating the set insertion step in Algorithm \ref{alg:abstractonecut}. Finally, by construction, we have $u^s_k=\mathcal{U}_{\A_{k,+}}(\tilde\lambda^{s})$ for some suitably chosen $\tilde{\lambda}^s$ with nonnegative entries. Hence, Algorithm \ref{alg:abstractonecut} achieves at least as much descent as the GCG step with properly chosen stepsize. While we emphasize that $u_k$ does not need to be admissible for the constrained problem and $M_\infty$ is not used in the method itself, we will see below that these descent guarantees allow to conclude that Algorithm \ref{alg:abstractonecut} inherits the characteristic global sublinear converge rate of GCG methods.     

For completeness, we mention that the algorithm of \cite{CriIglWal23} relies on a similar reformulation but employs the trivial constraint $\TV(u,\Omega) \leq M_{\TV}$ arising from bounds on the optimal cost which leads to GCG problems of the form
\begin{equation}
   v_k \in \argmin_{u} \bigg( -\int_{\Omega} p_k u \dd x + \TV(u, \Omega) \bigg) \quad \text{s.t.} \quad \TV(u, \Omega) \leq M_{\TV}.
\end{equation}
By the fundamental theorem of linear programming as well as a characterization of the extremal points of the TV-seminorm ball, we can select a solution of the form $v_k= M_{\TV} \1_{\bar{E}_k}$ where $\bar{E}_k$ is now obtained from 
\begin{equation} \label{eq:insertscaledintro}
    \argmax_{E \subset \Omega} \frac{\int_\Omega p_k \dd x}{\Per(E,\Omega)}.
\end{equation}
\subsection{Sublinear convergence} \label{subsec:sublinear}
In this section, we prove the sublinear convergence of Algorithm \ref{alg:abstractonecut}.
We require some preparatory results:
\begin{lemma} \label{lem:optifinite}
 There holds
 \begin{gather} \label{eq:optfinitedim}
   \int_{E^j_k} p_k \dd x= \Per(E^j_k,\Omega) \quad \text{for all} \quad E^j_k \in \A_k, \quad \text{as well as}\\
   0 \leq r_J(u_k) \leq r_{\A_k}(\lambda_k) \leq -M_{\infty}\bigg( -\int_{\bar{E}_k} p_k \dd x + \Per(\bar{E}_k, \Omega) \bigg), ~u_k=\mathcal U_{\A_k}(\lambda_k) \in E_J(0).\notag
 \end{gather}
\end{lemma}
\begin{proof}
The first statement follows immediately from deriving first-order necessary optimality conditions for the finite dimensional minimization problem in \eqref{eq:formulauk}  noting that $\lambda^j_k >0$, see \eqref{eq:formulauk}. Concerning the second, let $\bar{u}$ denote any minimizer of \eqref{def:BVprob} for which we recall that $\|\bar{u}\|_\infty \leq M_\infty$ holds. Now use \eqref{eq:optfinitedim} to estimate
\begin{align}
 0 \leq r_J(u_k) \leq r_{\A_k}(\lambda_k) &\leq \int_\Omega p_k (\bar{u} - u_k) \dd x - \TV(\bar{u},\Omega) + \sum_{E^j_k \in \A_k} \lambda^j_k \Per(E^j_k,\Omega)\\ & = \int_\Omega p_k \bar{u} \dd x - \TV(\bar{u},\Omega) \\
& = - \|\bar{u}\|_\infty \bigg( - \int_\Omega p_k \big( \bar{u}/\|\bar{u}\|_\infty \big) \dd x + \TV\big( \bar{u}/\|\bar{u}\|_\infty,\Omega \big) \bigg)\\
& \leq -M_\infty \min_{0 \leq u \leq 1} \bigg( -\int_\Omega p_k u \dd x + \TV(u,\Omega) \bigg)\\ 
& \leq -M_\infty \bigg( -\int_{\bar{E}_k} p_k \dd x + \Per(\bar{E}_k,\Omega) \bigg),\end{align}
where the second inequality follows from 
\begin{equation}
  \TV(u_k,\Omega) \leq \sum_{E^j_k \in \A_k} \lambda^j_k \TV(\1_{E^j_k},\Omega)=\sum_{E^j_k \in \A_k} \lambda^j_k \Per(E^j_k,\Omega), 
\end{equation}
the third is due to the convexity of $F$ and the definition of $p_k$, the fourth one follows from Lemma \ref{lem:uniformlinfty} and the final one is a consequence of the definition of $\bar{E}_k$. Finally, by construction, it holds
\begin{equation}
  r_J(u_k) \leq r_{\A_k}(\lambda_k) \leq r_{\A_k}(0)=r_J(0) 
\end{equation}
and thus $J(u_k) \leq J(0)$, i.e. $u_k \in E_J(0)$.
\end{proof}
\begin{lemma} \label{lem:compprodnegative}
Assume that $\bar{E}_k$ is decomposable as $\bar{E}_k = \bigcup^{n_k}_{m=1} \bar{E}^m_k$, and satisfies
\begin{equation}
      \bar{E}_k \in \argmin_{E \subset \Omega} \left \lbrack -\int_E p_k \dd x+ \Per(E,\Omega)\right\rbrack. 
    \end{equation}
    Then it holds 
    \begin{equation} \label{eq:compneg}
        -\int_{ \bar{E}^m_k} p_k \dd x+ \Per(\bar{E}^m_k,\Omega) \leq 0 \quad \text{for all } m=1, \dots,n_k.
    \end{equation}
\end{lemma}
\begin{proof}This follows exactly as the beginning of the proof of Lemma \ref{lem:upperlowerbounds}. 
\end{proof}
As a consequence, we can derive an upper bound on the per-iteration descent of Algorithm \ref{alg:abstractonecut} which can then be used to conclude the sublinear convergence of the method.
\begin{proposition}\label{prop:descent}
For all~$k \geq 1$, the iterates $u_k$ satisfy
\begin{equation} \label{eq:descentineq}
  r_{\A_{k+1}}(\lambda_{k+1})-r_{\A_k}(\lambda_k) \leq - \frac{r_{\A_k}(\lambda_k)}{2} \min \left\{1, \frac{r_{\A_k}(\lambda_k)}{L_{\nabla F}\|K\|^2 \big(M_q+ M_\infty |\Omega|^{\frac{1}{q}} \big)^2}\right\} \leq 0.
\end{equation}  
\end{proposition}
\begin{proof}
For $s \in [0,1]$, define $u^s_k= u_k+ s(M_\infty \1_{\bar{E}_k}-u_k)$ which we can rewrite as $u^s_k= \mathcal{U}_{\A_{k,+}}(\tilde\lambda^{s})$ and
\begin{align}
  u^s_k= (1-s) \sum_{E^j_k \in \A_k} \lambda^j_k\1_{E^j_k} + sM_\infty \1_{\bar{E}_k}= (1-s) \sum_{E^j_k \in \A_k} \lambda^j_k\1_{E^j_k} + sM_\infty \sum^{n_k}_{m=1} \1_{\bar{E}^{m}_k}
\end{align}
by choosing $\tilde\lambda^{s}$ suitably and noting that $\bar{E}^m_k \cap \bar{E}^n_k=\emptyset$, $m \neq n$. As a consequence,
\begin{equation}
  r_{\A_{k+1}}(\lambda_{k+1})-r_{\A_k}(\lambda_k)= r_{\A_{k,+}}(\lambda_{k,+})-r_{\A_k}(\lambda_k) \leq r_{\A_{k,+}}(\tilde\lambda^{s})-r_{\A_k}(\lambda_k)
\end{equation}
as well as
\begin{equation}
  r_{\A_{k,+}}(\tilde\lambda^{s})-r_{\A_k}(\lambda_k)= F(K u^s_k)-F(Ku_k)+ s \Bigg(M_\infty \Per(\Bar{E}_k, \Omega)-\sum_{E^j_k \in \A_k} \lambda^j_k \Per(E^j_k, \Omega) \Bigg)
\end{equation}
where we use $\sum^{n_k}_{m=1} \Per(\bar{E}^m_k ,\Omega)=\Per(\bar{E}_k, \Omega)$. By Taylor expansion of the difference, we further obtain
\begin{equation}
  F(Ku^s_k)- F(Ku_k) \leq s\int_\Omega p_k(u_k- M_\infty\1_{\bar{E}_k}) \dd x + \frac{s^2L_{\nabla F} \|K\|^2}{2} \|u_k- M_\infty \1_{\bar{E}_k}\|^2_{L^q}
\end{equation}
and thus, using Lemma \ref{lem:optifinite},
\begin{align} \label{eq:helpest}
  r_{\A_{k+1}}(\lambda_{k+1})-r_{\A_k}(\lambda_k) &\leq s M_\infty \left(-\int_{\bar{E}_k} p_k \dd x+ \Per(\bar{E}^\gamma_k,\Omega) \right)\\&\qquad+ \frac{s^2 L_{\nabla F} \|K\|^2}{2} \|u_k- M_\infty \1_{\bar{E}_k}\|^2_{L^q} \notag\\
  &\leq -s r_{\A_k}(\lambda_k)+ \frac{s^2 L_{\nabla F} \|K\|^2}{2} \|u_k- M_\infty \1_{\bar{E}_k}\|^2_{L^q}
\end{align}
Recalling that $u_k \in E_J(0)$, i.e. $\|u_k\|_{L^q} \leq M_q$, we now estimate
\begin{equation}
  \|u_k- M_\infty \1_{\bar{E}_k}\|^2_{L^q} \leq \left(M_q+ M_\infty |\Omega|^{\frac{1}{q}} \right)^2.
\end{equation}
Substituting this bound into \eqref{eq:helpest} and minimizing w.r.t $s \in [0,1]$, we find
\begin{align}
  \min_{s \in [0,1]} &\left\lbrack-s r_{\A_k}(\lambda_k)+ \frac{s^2 L_{\nabla F} \|K\|^2}{2} \left(M_q+ M_\infty |\Omega|^{\frac{1}{q}} \right)^2 \right \rbrack \\ &\leq - \frac{r_{\A_k}(\lambda_k)}{2} \min \left\{1, \,\frac{r_{\A_k}(\lambda_k)}{L_{\nabla F}\|K\|^2 \big(M_q+ M_\infty |\Omega|^{\frac{1}{q}} \big)^2}\right\}.
\end{align}
\end{proof}
\begin{theorem} \label{thm:sublinear}
Let $u_k$, $k\in \N$, be generated by Algorithm \ref{alg:abstractonecut}. Then there holds
\begin{align}
    r_J(u_k) \leq \frac{r_{\A_1}(\lambda_1)}{1+q(k-1)} \quad \text{where} \quad q=\frac{1}{2} \min \left\{1, \,\,\frac{r_{\A_1}(\lambda_{1})}{L_{\nabla F}\|K\|^2 \big(M_q+ M_\infty |\Omega|^{\frac{1}{q}} \big)^2}\right\}
\end{align}
as well as $y_k \rightarrow \bar{y}$ in $Y$ and $p_k \rightarrow \bar{p}$ in $L^d(\Omega)$.
Further, $u_k$ admits at least one strictly convergent subsequence and every strict accumulation point is a minimizer of \eqref{def:BVprob}.
\end{theorem}
\begin{proof}
Dividing \eqref{eq:descentineq} by $r_{\A_1}(\lambda_1)$ and noting that $r_{\A_k}(\lambda_k) \leq r_J(0)$, $k \in \N$, we obtain
\begin{align}
  \frac{r_{\A_{k+1}}(\lambda_{k+1})}{r_{\A_1}(\lambda_1)} \leq \frac{r_{\A_k}(\lambda_k)}{{r_{\A_1}(\lambda_1)}} - \frac{1}{2} \min \left\{1,\,\, \frac{r_{\A_1}(\lambda_1)}{L_{\nabla F}\|K\|^2 \big(M_q+ M_\infty |\Omega|^{\frac{1}{q}} \big)^2}\right\} \left( \frac{r_{\A_k}(\lambda_k)}{r_{\A_1}(\lambda_1)} \right)^2.
\end{align}
The claimed convergence rate then follows by \cite[Lemma 3.1]{Dunn80} as well as $r_J(u_k) \leq r_{\A_k}(\lambda_k)$.
The statement on strictly convergent subsequences of $u_k$ as well as the optimality of strict accumulation points follows by the same arguments as in \cite{CriIglWal23}. Finally, the convergence of the $p_k$ and $y_k$ follows by uniqueness of the optimal observation $\bar{y}$ and the weak-to-strong continuity of $K$.
\end{proof}
\begin{remark} \label{rem:sublinwithoutsplit}
It is worth noting that the splitting of $\bar{E}_k$ into indecomposable components at each step as well as a full resolution of the finite-dimensional coefficient problem are not necessary to achieve a sublinear rate of convergence as in Theorem \ref{thm:sublinear}. More precisely, a comparable result can be proven, mutatis mutandis, for sequences $u_k$ which merely satisfy
\begin{equation}
 r_{\A_{k+1}}(\lambda_{k+1}) \leq \min_{s \in [0,1]} r_{\A_{k,+}}(\tilde\lambda^{s}) ,
\end{equation}
with $\mathcal{U}_{\A_{k,+}}(\tilde\lambda^{s})=u_k+s(M_\infty \1_{\bar{E}_k}-u_k)$.
\end{remark}

\subsection{Linear convergence under structural assumptions} \label{subsec:linear}
In this section, we finally prove that Algorithm \ref{alg:abstractonecut} eventually converges linearly provided that the optimal dual variable $\bar{p}$ for Problem \eqref{def:BVprob} satisfies additional structural assumptions in the spirit of \cite{DecDuvPet24}. In order to profit from the tools developed in the latter, we restrict ourselves to the particular case of two-dimensional, i.e. $d=2$, and convex domains $\Omega$. This convexity, in combination with Assumption \eqref{ass:maxminawayfrombdy}, ensures that the insertion problem behaves essentially as in the case where $\Omega = \R^d$, see Appendix \ref{sec:appendix}.

We start by assuming that:
\begin{enumerate}[label=\textbf{B\arabic*}]
\item \label{ass:maxminawayfrombdy} The unique maximal solution $\bar{E}= \bigcup^N_{\gamma=1} \bar{E}^\gamma $ of Problem \eqref{def:prescribedcurvature} with $p=\bar{p}$ and indecomposable components $\bar{E}^\gamma$ satisfies $\dist(\bar{E}, \partial \Omega) >0$. Moreover, there holds $\dim \Span \left\{K\1_{\bar{E}^\gamma}\right\}^N_{\gamma=1}=N$ and 
\begin{equation}\label{eq:MCunique}
    \argmin \eqref{def:prescribedcurvature}= \{\emptyset\} \cup \Bigg\{E\;\Bigg\vert\; \exists\, \mathcal{I} \subset \{1,\dots,N\},~E=\bigcup_{\gamma \in \mathcal{I}} \bar{E}^\gamma  \,\Bigg\}.
\end{equation}
\end{enumerate}
We note that we consider the representative of $\bar{E}$ satisfying \eqref{eq:bdy} when talking about the quantity $\dist(\bar{E}, \partial \Omega)$.
Arguing along the lines of \cite[Proposition 3.5]{BreCarFanWal24}, this assumption implies that \eqref{def:BVprob} admits a unique solution $\bar u$ which is of the form $\bar{u}=\sum^N_{\gamma=1} \bar{\lambda}^\gamma \1_{\bar{E}^\gamma}$ for some unique weights $\bar{\lambda}^\gamma \geq 0$. As a consequence, see Theorem \ref{thm:sublinear}, we have $u_k \rightarrow \bar{u}$ in $L^q(\Omega)$. The following strict complementarity assumption is made:
\begin{enumerate}[label=\textbf{B\arabic*}]
\setcounter{enumi}{1}
\item \label{ass:strictcomplementarity} There exist $\bar{\lambda}^\gamma >0$ for $\gamma=1,\ldots,N$ such that the unique solution $\bar{u}$ of Problem \eqref{def:BVprob} is of the form $\bar{u}=\sum^N_{\gamma=1} \bar{\lambda}^\gamma \1_{\bar{E}^\gamma}$.
\end{enumerate}

We further require stronger regularity assumptions on the fidelity term $F$ as well as on the forward operator $K$:
\begin{enumerate}[label=\textbf{B\arabic*}]
\setcounter{enumi}{2}
\item \label{ass:strongconv} $F$ is strongly convex on a neighborhood $\mathcal{N}(\bar y)$ of $\bar{y}$, i.e. there is $\theta>0$ with
\begin{equation}
    (\nabla F(y_1)-\nabla F(y_2),y_1-y_2)_Y \geq \theta \|y_1-y_2\|^2_Y \quad \text{for all} \quad y_1,y_2 \in \mathcal{N}(\bar{y}).
\end{equation}
\item \label{ass:smoothduals} The adjoint operator $K^\ast$ maps continuously from $Y$ to $\cC^1(\cl \Omega)$.
\end{enumerate}
The assumption $\dist(\bar{E}, \partial \Omega)>0$ together with Assumption \eqref{ass:smoothduals} allows us to apply Proposition \ref{prop:globalize} which guarantees that $\bar{E}$ is the maximal minimizer of a prescribed mean curvature problem over the whole space $\R^2$. Since as seen in Lemma \ref{lem:upperlowerbounds} indecomposable components of minimizers of \eqref{def:prescribedcurvature} remain minimizers, using the regularity results recalled in Proposition \ref{prop:existenceMC} we get that the components $\bar{E}^\gamma$ have $\cC^2$ boundaries. Assumption \eqref{ass:smoothduals} also implies that $p_k \rightarrow \bar{p}$ in $\cC^1(\cl \Omega)$. The main idea in the following is to use this stronger convergence together with \eqref{ass:maxminawayfrombdy} to interpret the new candidate set $\bar{E}_k$ from Algorithm \ref{alg:abstractonecut} as a smooth deformation of a subset of $\bar{E}$, see Theorem \ref{thm:deformationinsert} below. We emphasize that $\dist(\bar{E}, \partial \Omega)>0$ is crucial to be able to perform these arguments. On the other hand, solutions of \eqref{def:prescribedcurvature} on the whole space with $p \in L^d(\R^d)$ are known (see \cite[Lemma 4]{ChaDuvPeyPoo17}) to be bounded, so in cases where the solutions do not strongly depend on the boundary conditions, one may choose $\Omega$ large enough such that this condition holds.

In order to quantify these observations, we rely on the following stability property:
\begin{enumerate}[label=\textbf{B\arabic*}]
\setcounter{enumi}{4}
\item \label{ass:quadraticgrowth} We have the following quadratic growth condition: There is $\eps_0 >0$ with
\begin{equation}\label{eq:quadgrowthmax} -\int_{\varphi_{\#}(\bar{E})} \bar{p} \dd x + \Per\big(\varphi_{\#}(\bar{E}),\Omega\big) \geq -\int_{\bar{E}} \bar{p} \dd x + \Per(\bar{E},\Omega) + \kappa \|\varphi\|^2_{H^1(\partial \bar{E})} \end{equation}
for all $\varphi \in \cC^2(\partial \bar{E})$ with $\|\varphi\|_{\cC^2(\partial \bar{E})} \leq \eps_0$.
\end{enumerate}
The quadratic growth assumption in \eqref{ass:quadraticgrowth} might be quite opaque, since it involves the $H^1$ norm of deformations. Let us point out that easier to check conditions with Hessians implying such quadratic growth are known, as formulated in \cite{DecDuvPet24} which in turn makes use of the stability results for shape optimization of \cite[Thm.~1.1]{DamLam19}. In particular, \cite[Prop.~D.2]{DecDuvPet24} provides a sufficient condition in terms of the mean curvature $H_{\bar{E}}$ of the boundary of $\bar{E}$ and its outer normal vector $n_{\bar{E}}$:
\[\sup_{x \in \partial \bar{E}} \left\lbrack H_{\bar{E}}(x)^{2} + \frac{\partial \bar{p}}{\partial n_{\bar{E}}}(x)\right\rbrack < 0.\]
Moreover, if $\bar{p}$ satisfies this condition and additionally $\bar{p}(x)=H_{\bar{E}}(x)$ at all $x \in \partial \bar{E}$, then \eqref{eq:MCunique} is also satisfied. Finally, we note that the analogous condition formulated on each $\bar{E}^\gamma$ automatically follows from \eqref{ass:quadraticgrowth}.

We further emphasize that while some of these assumptions, e.g. \eqref{ass:strongconv} and \eqref{ass:smoothduals}, can be checked a priori in simple settings, the more technical ones can only be verified a posteriori once $\bar{p}$ and $\bar{u}$ are computed, \eqref{ass:strictcomplementarity}, and, in the case of the condition \eqref{ass:maxminawayfrombdy} on the maximal minimizer of \eqref{def:prescribedcurvature} with $\bar{p}$, would require additional approximations \cite{ButCarCom07} and involved numerical computations. Moreover, the practical realization of Algorithm \ref{alg:abstractonecut} often requires an additional discretization of the problem, adding another level of complexity to the problem. For example, after approximating elements in $L^q(\Omega)$ by piecewise constant functions on a triangulation $\mathcal{T}$ of $\Omega$, \cite{CriIglWal23} proves finite-step convergence of a discretized algorithm owing to the fact that the set of triangulated sets $\mathcal{S}_{\mathcal{T}}(\Omega)$ in $\Omega$ is finite.

As a consequence, the presented result should be understood as a first step towards understanding the practical efficiency of accelerated conditional gradient-like methods for TV-regularization and leaves room for further work. The following remark summarizes some relaxations of the presented assumptions which, while interesting, would require additional technical work and are, consequently, disregarded at the moment to strike a balance between generality and readability.
\begin{remark}\label{rem:ndsc}
Assumptions \eqref{ass:strictcomplementarity} and \eqref{ass:quadraticgrowth} could be relaxed to a setting analogous to the one imposed by the non-degenerate source condition of \cite[Def.~1]{DecDuvPet24}. In that case, instead of \eqref{eq:MCunique} one would prescribe that all possible solutions of \eqref{def:prescribedcurvature} arise from a collection of simple sets corresponding to the decompositions of all level sets of $\bar{u}$, and the quadratic growth assumption \eqref{ass:quadraticgrowth} would have to be formulated around each of these sets. We stay in the more restricted setting for clarity and brevity, but our analysis would follow along similar lines, provided that the decomposition step of Algorithm \ref{alg:abstractonecut} would be replaced by finding all components of both $\bar{E}_k$ and $\Omega \setminus \bar{E}_k$, and adding all of them to the active set. Furthermore, given \eqref{ass:maxminawayfrombdy}, Assumption \eqref{ass:smoothduals} could be weakened to requiring interior regularity $K^* y \in \mathcal{C}^1(\Omega_o) $ for some subset $\Omega_o \subset \Omega$ with $\bar{E} \subset \Omega_o$. Finally, we point out that we see no clear obstacles to extending the result to higher dimensions (up to at most $d=7$ for regularity results to apply, see Proposition \ref{prop:existenceMC} and the references therein), but we stay in $d=2$ to directly use the stability results for minimizers of \eqref{def:prescribedcurvature} in the form stated in \cite{DecDuvPet24}.
\end{remark}

Given \eqref{ass:maxminawayfrombdy}-\eqref{ass:quadraticgrowth}, we will prove that Algorithm \ref{alg:abstractonecut} converges linearly in the asymptotic regime, i.e. there is $\bar{k} \geq 1$ and $\zeta \in (0,1)$ such that
\begin{equation}
 r_{\A_{k+1}}(\lambda_{k+1}) \leq \zeta r_{\A_k}(\lambda_k), \quad  r_J(u_k) \leq C_{\text{lin}} \zeta^k \quad \text{for all} \quad k \geq \bar{k}.
\end{equation}
For this purpose, we want to proceed analogously to Theorem \ref{thm:sublinear} and estimate the per-iteration decrease of Algorithm \ref{alg:abstractonecut} via a surrogate $\widehat{u}^s_k$ for which the former is easy to quantify. Our considerations rest on the following result characterizing the set $\bar{E}_k$:

\begin{theorem} \label{thm:deformationinsert}For every $\epsilon >0$ there is $\eta >0$ such that if $\|p_k - \bar{p}\|_{C^1(\cl \Omega)} < \eta$, then for every solution $\bar{E}_k$ of \eqref{def:prescribedcurvature} with $p=p_k$ there is an index set $\mathcal{I}_k \subset \{1, \dots,N\}$ and a deformation
\[ \bar\varphi_k \in \cC^2\big(B_{\mathcal{I}_k},\, \R^2\big) \quad\text{for}\ \,B_{\mathcal{I}_k}:=\bigcup_{\gamma \in \mathcal{I}_k} \partial \bar{E}^\gamma \quad \text{with}\ \, \|\bar\varphi_k\|_{\cC^2(B_{\mathcal{I}_k})}<\eps\]
such that
\begin{equation}
  \bar{E}_k = \bigcup_{\gamma \in \mathcal{I}_k} \bar{E}^\gamma_k \quad \text{where} \quad \bar{E}^\gamma_k= (\bar{\varphi}_k)_{\#}\big(\bar{E}^\gamma\big).
\end{equation}
\end{theorem}
\begin{proof}
By Assumption \eqref{ass:maxminawayfrombdy}, we recall that there exists a smooth extension $\widehat{p}$ of $\bar{p}$ to the full space such that $\bar{E}$ is the maximal minimizer of the corresponding prescribed mean curvature problem. Due to Assumption \eqref{ass:smoothduals} as well as Proposition \ref{prop:globalize}, the same holds true for $p_k$, i.e. we find smooth extensions $\widehat{p}_k$ such that $\bar{E}_k$ is a minimizer for the global problem and we have $\widehat{p}_k \rightarrow \widehat{p}$ in $L^2(\R^2) \cap \cC^1(\cl \Omega)$. The desired deformation property then follows from \cite[Prop.~4.1]{DecDuvPet24}.
\end{proof}
For each $k$ large enough and given the associated index set $\mathcal{I}_k$ from Theorem \ref{thm:deformationinsert}, we set $\bar{E}_{\mathcal{I}_k}=\bigcup_{\gamma \in \mathcal{I}_k} \bar{E}^\gamma$ for abbreviation.
The construction of an improved function $\widehat{u}^s_k$ proceeds along the following outline: 
\begin{enumerate}
  \item For large $k$, the active set $\A_k$ decomposes into $N$ disjoint clusters $\A^\gamma_k$ such that each $E \in \A^\gamma_k$ is a close, smooth deformation of the corresponding optimal set $\bar{E}^\gamma$, see Lemma \ref{lem:noticludedlargek} and Corollary \ref{coroll:splitactive}.
  \item We estimate the difference between $u_k$ and $\bar{u}$, measured in terms of the weighted sums of the norms of the corresponding deformations, by powers of the residual $r_{\A_k}(\lambda_k)$. We proceed similarly for the distance between the candidate set $\bar{E}_k$ and corresponding subsets of $\bar{E}$.  
  \item Summarizing the previous steps, the iterate $u_k$ can be represented as
  \begin{equation}
    u_k = \sum^N_{\gamma=1} \sum_{E^{\gamma,\ell}_k \in \A^\gamma_k} \lambda^{\gamma,\ell}_k \1_{E^{\gamma,\ell}_k}
  \end{equation}
  with $ \lambda^{\gamma,\ell}_k>0$. Exploiting the clustered structure of $\A_k$, we finally obtain the surrogate $\widehat{u}^k_s$ via localized convex combinations $\widehat{u}^s_k\coloneqq u_k + s(\widehat{v}_k-u_k)$, where
  \begin{equation}
   \widehat{v}_k=\sum_{\gamma \not\in \mathcal{I}_k} \sum_{E^{\gamma,\ell}_k \in \A^\gamma_k} \lambda^{\gamma,\ell}_k \1_{E^{\gamma,\ell}_k} + \sum_{\gamma \in \mathcal{I}_k} \Bigg( \sum_{E^{\gamma,\ell}_k \in \A^\gamma_k} \lambda^{\gamma,\ell}_k \Bigg) \1_{\bar{E}^\gamma_k},
  \end{equation}
  which partially lump the contributions of several clusters into that of one single set per cluster, while keeping the others unchanged. This local update, which stands in contrast with the global update $u^k_s$ in Theorem \ref{thm:sublinear}, allows for a refined analysis of the per-iteration descent, eventually leading to linear convergence.
\end{enumerate}

\subsubsection{Preparatory results} \label{subsubsec:prepraratorymain}
We begin by showing a geometric consequence of our assumptions on $K$ and $\bar{E}$: 
\begin{lemma}There is $\eps_0 >0$ depending on $\bar{E}$ such that the deformation-Lipschitz property
\begin{equation}\label{eq:deflipschitz}
  \left\|K\big(\1_{\varphi_\# (\bar{E}^{\gamma})}- \1_{\bar{E}^{\gamma}}\big)\right\|_Y \leq C_{K,\bar{E}}\|\varphi\|_{\bH}
\end{equation}
holds for all $\varphi \in \cC^2(\partial \bar{E})$ with $\|\varphi\|_{\cC^2(\partial \bar{E})} \leq \eps_0$.
\end{lemma}
\begin{proof} As already noted, Assumptions \eqref{ass:maxminawayfrombdy} and \eqref{ass:smoothduals} imply that the boundaries $\partial \bar{E}^{\gamma}$ can be parametrized as closed planar $\cC^2$ curves. In that situation and assuming that $\|\varphi\|_{L^\infty(\partial \bar{E}^{\gamma})}$ is smaller than both $\dist(\bar{E}, \partial \Omega)$ and the cut locus of $\partial \bar{E}^{\gamma}$ so that parallel curves remain $\cC^2$, we can use the relevant version of Weyl's tube formula (see \cite[Sec.~1.2]{Gra04}) to estimate the left hand side as
\begin{align}
    \left\|K\big(\1_{\varphi_\# (\bar{E}^{\gamma})}- \1_{\bar{E}^{\gamma}}\big)\right\|_Y &= \sup_{\|y\|_Y \leq 1} \int_\Omega K^\ast y\, \big(\1_{\varphi_\# (\bar{E}^{\gamma})}- \1_{\bar{E}^{\gamma}}\big) \dd x \\ &\leq \|K^\ast\| \big| \varphi_\# (\bar{E}^{\gamma}) \Delta \bar{E}^{\gamma}\big| \leq 2 \|K^\ast\| \Per(\bar{E}^{\gamma}, \R^2) \|\varphi\|_{L^\infty(\partial \bar{E}^{\gamma})}.
\end{align}
Using that $\Per(\bar{E}^{\gamma}, \R^2) \leq \Per(\bar{E}, \R^2)$ and a Sobolev inequality for the one-dimensional $\bar{E}^{\gamma}$, we obtain \eqref{eq:deflipschitz}.
\end{proof}

Recall the abbreviations $y_k=Ku_k,\bar{y}=K\bar{u}$, $p_k = -K^* \nabla F(Ku_k), \bar{p} = -K^* \nabla F(K\bar{p})$ as well as $y_k \rightarrow \bar{y}$ in $Y$ according to Theorem \ref{thm:sublinear}. We now show that all sets in $\A_k$ are deformations of optimal ones for large $k$.
\begin{lemma} \label{lem:noticludedlargek}
Let $\eps_0$ be given. There is $k_0 \in \N$ such that for all $k \geq k_0$ and all $E \in \A_k$, there is $j \in \{1, \dots, N\}$ such that
\begin{equation}
   \exists \varphi \in \cC^2(\partial \bar{E}^\gamma) \colon E= \varphi_\# \big(\bar{E}^\gamma\big), \quad \|\varphi\|_{\cC^2(\partial \bar{E}^\gamma)} \leq \eps_0.
\end{equation}
\end{lemma}
\begin{proof}
Because $\|p_k - \bar{p}\|_{\cC^1(\Omega)} \to 0$, using \cite[Prop.~4.1]{DecDuvPet24} we get that there is an index $k_0 \in \N$ such that
\begin{equation}
 E \in \A_k \setminus \A_{k_0} \Rightarrow \text{there are } j \text{ and } \varphi \in \cC^2(\partial \bar{E}^\gamma) \text{ with } E= \varphi_\# \big(\bar{E}^\gamma\big), \ \|\varphi\|_{\cC^2(\partial \bar{E}^\gamma)} \leq \eps_0.
\end{equation}
It remains to check that sets which do not satisfy the assumption will eventually be deleted. Therefore, assume that $E \in \A_k$ for all $k \in \N$ large enough. Then we have
\begin{equation}
 -\int_E \bar{p} \dd x+ \Per(E,\Omega) = \lim_{k \rightarrow \infty}-\int_E p_k \dd x+ \Per(E,\Omega)=0.
\end{equation}
This tells us that in fact, $E$ is a minimizer of 
\[\tilde{E} \mapsto -\int_{\widetilde{E}} \bar{p} \dd x + \Per(\widetilde{E},\Omega),\]
which is the functional of which $\bar{E}$ is the maximal minimizer. Since by the definition of the insertion step $E$ is indecomposable and the $\bar{E}^\gamma$ were defined as the indecomposable components of $\bar{E}$, there needs to be some $j \in \{1,\ldots,N\}$ such that $E = \bar{E}^\gamma$.
\end{proof}
\begin{corollary} \label{coroll:splitactive}
  There is some $\eps_0 >0$ such that for each $\eps < \eps_0$ we can find $k \in \N$ for which 
  \begin{gather}
    \A_k = \bigcup^N_{\gamma=1} \A^\gamma_k, \quad \A^\gamma_k \neq \emptyset, \quad \A^\gamma_k \cap \A^\nu_k =\emptyset \text{ if }\gamma \neq \nu, \notag \\
    E^{\gamma,\ell}_k \in \A^\gamma_k \Rightarrow \exists \varphi^{\gamma,\ell}_k \colon  E^{\gamma,\ell}_k= \left( \varphi^{\gamma,\ell}_k\right)_\# \big(\bar{E}^\gamma\big), \quad \|\varphi^{\gamma,\ell}_k\|_{\cC^2(\partial \bar{E}^\gamma)} \leq \eps \label{eq:componentsclose}
  \end{gather}
  and $\dist(E^{\gamma,\ell}_k, \partial \Omega) > 0$ for all $E^{\gamma,\ell}_k \in \A^\gamma_k$.
  \end{corollary}
\begin{proof}
The existence of the $\A^\gamma_k$ satisfying \eqref{eq:componentsclose} and with $\A_k = \bigcup^N_{\gamma=1} \A^\gamma_k$ follows directly from Lemma \ref{lem:noticludedlargek}. We only need to prove that $\A^\gamma_k \cap \A^\nu_k =\emptyset$ for $\gamma \neq \nu$. This follows readily by setting $\eps_0$ small enough, since otherwise for each $\eps >0$ small enough we would be able to find $\widehat{\varphi}$ so that $\bar{E}^\nu= \widehat{\varphi}_\# \big(\bar{E}^\gamma\big)$ and $\|\widehat{\varphi}\|_{\cC^2(\partial \bar{E}^\gamma)} \leq 2\eps$, which is impossible. Reducing $\eps_0$ further if necessary, \eqref{eq:componentsclose} together with the assumption $\dist(\bar{E}, \partial \Omega) > 0$ contained in \eqref{ass:maxminawayfrombdy} implies $\dist(E^{\gamma,\ell}_k, \partial \Omega) > 0$ for all $E^{\gamma,\ell}_k \in \A^\gamma_k$.
\end{proof}

Together with Lemma \ref{lem:optifinite}, we thus conclude that for every $\eps < \eps_0$ we can find $k \in \N$ such that there are $\lambda^{\gamma,\ell}_k>0$ with
\begin{equation}\begin{gathered}
      u_k = \sum^N_{\gamma=1} \sum_{E^{\gamma,\ell}_k \in \A^\gamma_k} \lambda^{\gamma,\ell}_k \1_{E^{\gamma,\ell}_k} =\sum^N_{\gamma=1} \sum_{E^{\gamma,\ell}_k \in \A^\gamma_k} \lambda^{\gamma,\ell}_k \1_{\big(\varphi^{\gamma,\ell}_k\big)_\#\big(\bar{E}^\gamma\big)}, \\\int_{E^{\gamma,\ell}_k} p_k \dd x= \Per(E^{\gamma,\ell}_k,\Omega), \quad \|\varphi^{\gamma,\ell}_k\|_{\cC^2(\partial \bar{E}^\gamma)} \leq \varepsilon.
\end{gathered}\end{equation}
The following lemma provides uniform bounds on the lumped sum of the coefficients associated to each cluster.
\begin{lemma} \label{lem:uniformboundslump}
There are constants $m_a,m_b>0$ for which we have 
\begin{equation}\label{eq:lumpbounds}
m_a \leq \sum_{E^{\gamma,\ell}_k \in \A^\gamma_k} \lambda^{\gamma,\ell}_k \leq m_b\quad \text{for all } k \text{ large enough}.
\end{equation}
\end{lemma}
\begin{proof}
Due to the separation of the components $\bar{E}^\gamma$ as well as due to \eqref{eq:componentsclose} we can find $\varepsilon>0$ small enough such that the sets
\begin{equation}
    \mathcal{N}(\bar{E}^\gamma)= \{x \in \Omega\, \vert\, \dist(x, \bar{E}^\gamma) \leq \varepsilon \} \subset \Omega
\end{equation}
are disjoint and we have ${E}^{\gamma,\ell}_k \subset \mathcal{N}(\bar{E}^\gamma)$ for all ${E}^{\gamma,\ell}_k \in  \A^{\gamma}_k$. Fix functions $\zeta^\gamma \in \cC^\infty$ such that for all $k$ large enough we have
\[\zeta^\gamma \equiv 1 \text{ on }\mathcal{N}(\bar{E}^\gamma), \quad \supp \zeta^\gamma \cap \mathcal{N}(\bar{E}^\nu) = \emptyset \text{ for all }\nu\neq \gamma.\]
Now, we test the weak convergence $u_k \wkto \bar{u}$ with he strongly convergent sequence $p_k \zeta^\gamma$ to obtain that
\[\sum_{E^{\gamma,\ell}_k \in \A^\gamma_k} \lambda^{\gamma,\ell}_k \Per\big({E}^{\gamma,\ell}_k,\Omega\big)=\int_{\Omega} p_k \zeta^\gamma u_k ~\dd x \xrightarrow[k \to \infty]{}\int_{\Omega} \bar{p} \zeta^\gamma \bar{u} ~\dd x =\bar{\lambda}^j \Per\big(\bar{E}^\gamma,\Omega\big).\]
Using Proposition \ref{prop:globalize}, Lemma \ref{lem:upperlowerbounds}, and the assumption that $\bar\lambda^\gamma > 0$ for all $\gamma$, \eqref{ass:strictcomplementarity}, we directly obtain \eqref{eq:lumpbounds}.
\end{proof}
Summarizing these observations, we are able to derive an estimate of the distance between $u_k$ and $\bar{u}$ measured by weighted norms of the deformations.     
\begin{lemma}\label{lem:ratestateadjoint}
For all $k\in \N$ large enough, there holds
\begin{align}
  \|y_k-\bar{y}\|_Y + \sum_{\gamma \in \mathcal{I}_k} \sum_{E^{\gamma,\ell}_k \in \A^\gamma_k} \lambda^{\gamma,\ell}_k \|\varphi^{\gamma,\ell}_k\|_{\bH} \leq c \sqrt{r_{\A_k}(\lambda_k)}
\end{align}
\end{lemma}
\begin{proof}
  Since $y_k \rightarrow \bar{y}$, there holds $y_k \in \mathcal{N}(\bar{y})$ for all $k \in \N$ large enough. By strong convexity of $F$ on $\mathcal{N}(\bar{y})$, we have
  \begin{equation}
      \begin{aligned}
       r_{\A_k}(\lambda_k) &\geq \theta \|y_k-\bar{y}\|^2_Y+ \int_\Omega \bar p(\bar{u}-u_k) \dd x -\TV(\bar{u})+ \sum^N_{\gamma=1} \sum_{E^{\gamma,\ell}_k \in \A^\gamma_k} \lambda^{\gamma,\ell}_k \Per(E^{\gamma,\ell}_k,\Omega) \\ &= \theta \|y_k-\bar{y}\|^2_Y+   \sum^N_{\gamma=1} \sum_{E^{\gamma,\ell}_k \in \A^\gamma_k} \lambda^{\gamma,\ell}_k \left(- \int_{E^{\gamma,\ell}_k} \bar{p} \dd x+ \Per(E^{\gamma,\ell}_k,\Omega)\right) 
      \end{aligned}
  \end{equation}
  where the equality follows by Proposition \ref{prop:firstorderopt}. In order to estimate the second term, note that, according to Corollary \ref{coroll:splitactive}, we have $\|\varphi^{\gamma,\ell}_k\|_{\bH} \leq \|\varphi^{\gamma,\ell}_k\|_{\cC^2(\partial \bar{E}^\gamma)} \leq \eps_0  $ for $k$ large enough. Hence, using \eqref{ass:quadraticgrowth}, we obtain
\begin{equation}
 \begin{aligned}
 \sum^N_{\gamma=1} \sum_{E^{\gamma,\ell}_k \in \A^\gamma_k} \lambda^{\gamma,\ell}_k &\left(- \int_{E^{\gamma,\ell}_k} \bar{p} \dd x + \Per(E^{\gamma,\ell}_k,\Omega)\right) \geq \kappa \sum^N_{\gamma=1} \sum_{E^{\gamma,\ell}_k \in \A^\gamma_k} \lambda^{\gamma,\ell}_k \left(\|\varphi^{\gamma,\ell}_k\|^2_{\bH}\right) \\ & \geq \kappa \sum_{\gamma \in \mathcal{I}_k} \sum_{E^{\gamma,\ell}_k \in \A^\gamma_k} \lambda^{\gamma,\ell}_k \left(\|\varphi^{\gamma,\ell}_k\|^2_{\bH} \right) \\
 & \geq \frac{\kappa}{\sum_{\gamma \in \mathcal{I}_k} \sum_{E^{\gamma,\ell}_k \in \A^\gamma_k} \lambda^{\gamma,\ell}_k} \Bigg(\sum_{\gamma \in \mathcal{I}_k} \sum_{E^{\gamma,\ell}_k \in \A^\gamma_k} \lambda^{\gamma,\ell}_k \|\varphi^{\gamma,\ell}_k\|_{\bH} \Bigg)^2\end{aligned}
\end{equation}
where the last step follows from Jensen inequality. Noting that 
\begin{equation}
\sum_{\gamma \in \mathcal{I}_k} \sum_{E^{\gamma,\ell}_k \in \A^\gamma_k} \lambda^{\gamma,\ell}_k \leq |\mathcal{I}_k| m_b \leq N m_b
\end{equation}
by Lemma \ref{lem:uniformboundslump}, the claimed statement follows. 
\end{proof}
By similar arguments, we quantify the distance between $\bar{E}_k$ and $\bar{E}$.
\begin{lemma} \label{lem:estofinsetset}
For all $k$ large enough, there holds
\begin{equation}
    \sum_{\gamma \in \mathcal{I}_k} \Bigg( \sum_{E^{\gamma,\ell}_k \in \A^\gamma_k} \lambda^{\gamma,\ell}_k \Bigg) \|\bar \varphi_k\|_{\bH} \leq c \sqrt{r_{\A_k}(\lambda_k)}.
\end{equation}
\end{lemma}
\begin{proof}
In view of Theorem \ref{thm:deformationinsert}, the same proof strategy as in Lemma \ref{lem:ratestateadjoint} can be applied, leading to 
\begin{equation}
\begin{aligned}
 \frac{\kappa}{\sum_{\gamma \in \mathcal{I}_k} \sum_{E^{\gamma,\ell}_k \in \A^\gamma_k} \lambda^{\gamma,\ell}_k} & \Bigg(\sum_{\gamma \in \mathcal{I}_k} \Bigg( \sum_{E^{\gamma,\ell}_k \in \A^\gamma_k} \lambda^{\gamma,\ell}_k \Bigg) \|\bar \varphi_k\|_{\bH} \Bigg)^2 \\&\leq \sum_{\gamma \in \mathcal{I}_k} \Bigg( \sum_{E^{\gamma,\ell}_k \in \A^\gamma_k} \lambda^{\gamma,\ell}_k \Bigg) \bigg( -\int_{\bar{E}^\gamma_k} \bar{p} \dd x+ \Per(\bar{E}^\gamma_k, \Omega) \bigg) \\ & \leq 
 m_b \sum_{\gamma \in \mathcal{I}_k} \bigg( -\int_{\bar{E}^\gamma_k} \bar{p} \dd x+ \Per(\bar{E}^\gamma_k, \Omega) \bigg)\end{aligned}
\end{equation}
where we use Lemma \ref{lem:uniformboundslump} as well as the fact that $\int_E \bar{p} \dd x \leq \Per(E,\Omega)$ for all $E \subset \Omega$, see Proposition \ref{prop:firstorderopt}, in the final inequality. Note that
\begin{align}
\sum_{\gamma \in \mathcal{I}_k} \bigg( -\int_{\bar{E}^\gamma_k} \bar{p} \dd x + &\Per(\bar{E}^\gamma_k, \Omega) \bigg) = -\int_{\bar{E}_k} \bar{p} \dd x+ \Per(\bar{E}_k,\Omega) \\ & = -\int_{\bar{E}_k} \bar{p} \dd x+ \Per(\bar{E}_k,\Omega)+ \int_{\bar{E}_{\mathcal{I}_k}} \bar{p} \dd x- \Per(\bar{E}_{\mathcal{I}_k},\Omega) \\ & \leq -\int_\Omega (\bar{p}-p_k)\left(\1_{\bar{E}_k}-\1_{\bar{E}_{\mathcal{I}_k}} \right) \dd x,
\end{align}
where the second equality uses $-\int_{\bar{E}_{\mathcal{I}_k}} \bar{p} \dd x= \Per(\bar{E}_{\mathcal{I}_k},\Omega)$ and the final inequality is due to minimality of $\bar{E}_k$, i.e.
\begin{equation}
    -\int_{\bar{E}_k} p_k \dd x + \Per(\bar{E}_k,\Omega)\leq -\int_{\bar{E}_{\mathcal{I}_k}} p_k \dd x + \Per(\bar{E}_{\mathcal{I}_k},\Omega) .
\end{equation}
Now, we further estimate
\begin{align}
    \int_\Omega (\bar{p}-p_k)\left(\1_{\bar{E}_k}-\1_{\bar{E}_{\mathcal{I}_k}} \right) \dd x &=\left(\nabla F(\bar{y})-\nabla F(y_k), K\left(\1_{\bar{E}_k}-\1_{\bar{E}_{\mathcal{I}_k}} \right)\right)_Y \\ &\leq L_{\nabla F} \|y_k-\Bar{y}\|_Y \left\|K\left(\1_{\bar{E}_k}-\1_{\bar{E}_{\mathcal{I}_k}} \right)\right\|_Y \\&\leq c \left\|K\left(\1_{\bar{E}_k}-\1_{\bar{E}_{\mathcal{I}_k}} \right)\right\|_Y \sqrt{r_{\A_k}(\lambda_k)}. 
\end{align}
using again Lemma \ref{lem:ratestateadjoint} in the final estimate. Finally, the claim follows due to
\begin{equation}
\begin{aligned}
  \left\| K\left( \1_{\bar{E}_k}-\1_{\bar{E}_{\mathcal{I}_k}} \right)\right\|_Y & \leq \sum_{\gamma \in \mathcal{I}_k} \left\| K\left( \1_{\bar{E}^\gamma_k}-\1_{\bar{E}^\gamma} \right)\right\|_Y \\
  &\leq \frac{C_{K, \bar E}}{m_a}\sum_{\gamma \in \mathcal{I}_k} \Bigg( \sum_{E^{\gamma,\ell}_k \in \A^\gamma_k} \lambda^{\gamma,\ell}_k \Bigg) \left\| K\left( \1_{\bar{E}^\gamma_k}-\1_{\bar{E}^\gamma} \right)\right\|_Y \\ & \leq \frac{C_{K, \bar E}}{m_a}\sum_{\gamma \in \mathcal{I}_k} \Bigg( \sum_{E^{\gamma,\ell}_k \in \A^\gamma_k} \lambda^{\gamma,\ell}_k \Bigg) \|\bar{\varphi}_k\|_{\bH}.
\end{aligned}
\end{equation}
invoking Lemma \ref{lem:uniformboundslump} and the deformation-Lipschitz property.
\end{proof}
\subsubsection{Proof of the main result}\label{subsubsec:linearconvergence}
We can now prove the asymptotic linear convergence of Algorithm \ref{alg:abstractonecut}. For this purpose, and for all $k \in \N$ large enough, set 
\begin{equation}
  \widehat{v}_k\coloneqq \sum_{\gamma \not\in \mathcal{I}_k} \sum_{E^{\gamma,\ell}_k \in \A^\gamma_k} \lambda^{\gamma,\ell}_k \1_{E^{\gamma,\ell}_k} + \sum_{\gamma \in \mathcal{I}_k} \Bigg( \sum_{E^{\gamma,\ell}_k \in \A^\gamma_k} \lambda^{\gamma,\ell}_k \Bigg) \1_{\bar{E}^\gamma_k}, \quad \widehat{u}^s_k\coloneqq u_k + s(\widehat{v}_k-u_k)
\end{equation}
for all $s\in[0,1]$.
The following lemma summarizes some properties of these objects:
\begin{lemma} \label{lem:propofwidehat}
For all $k\in\N$ large enough, there holds
\begin{equation}
    \int_\Omega p_k (\widehat{v}_k-u_k) \dd x= \sum_{\gamma \in \mathcal{I}_k}\Bigg( \sum_{E^{\gamma,\ell}_k \in \A^\gamma_k} \lambda^{\gamma,\ell}_k \Bigg) \int_{\bar{E}^\gamma_k} p_k \dd x- \sum_{\gamma \in \mathcal{I}_k}\sum_{E^{\gamma,\ell}_k \in \A^\gamma_k} \lambda^{\gamma,\ell}_k \Per(E^{\gamma,\ell}_k,\Omega)
\end{equation}
as well as
\begin{equation}
 r_{\A_{k+1}}(\lambda_{k+1})=r_{\A_{k,+}}(\lambda_{k,+}) \leq r_{\A_{k,+}}(\lambda_{k,s}), \quad    \|K(\widehat{v}_k-u_k)\|_Y \leq C_{\mathcal{D}} \sqrt{r_{\A_k}(\lambda_k)}.
\end{equation}
for some $C_{\mathcal{D}}>0$ and $\lambda_{k,s} \geq 0$ such that $\widehat{u}^s_k= \mathcal{U}_{\A_{k,+}}(\lambda_{k,s})$.
\end{lemma}
\begin{proof}
The first statement follows directly by definition of $\widehat{v}_k$ as well as Lemma \ref{lem:optifinite}. Next, we start by estimating
\begin{align}
  &\left\| K\left( \widehat{v}_k-u_k \right)\right\|_Y \\ &\quad\leq\sum_{\gamma \in \mathcal{I}_k} \Bigg\lbrack \Bigg( \sum_{E^{\gamma,\ell}_k \in \A^\gamma_k} \!\!\lambda^{\gamma,\ell}_k \Bigg) \left\| K\left( \1_{\bar{E}^\gamma_k}-\1_{\bar{E}^\gamma} \right)\right\|_Y+ \sum_{E^{\gamma,\ell}_k \in \A^\gamma_k} \!\!\lambda^{\gamma,\ell}_k \left\| K\left( \1_{E^{\gamma,\ell}_k}-\1_{\bar{E}^\gamma} \right)\right\|_Y \Bigg \rbrack.
\end{align}
The deformation-Lipschitz property \eqref{eq:deflipschitz}, Lemma \ref{lem:ratestateadjoint} and Lemma \ref{lem:estofinsetset} imply
\begin{align}
    &\sum_{E^{\gamma,\ell}_k \in \A^\gamma_k} \!\! \lambda^{\gamma,\ell}_k \left\| K\left( \1_{E^{\gamma,\ell}_k}-\1_{\bar{E}^\gamma} \right)\right\|_Y \\ &\qquad\leq C_{K, \bar E} \sum_{\gamma \in \mathcal{I}_k} \sum_{E^{\gamma,\ell}_k \in \A^\gamma_k} \!\! \lambda^{\gamma,\ell}_k \|\varphi^{\gamma,\ell}_k\|_{\bH} \leq C \sqrt{r_{\A_k}(\lambda_k)}
\end{align}
as well as
\begin{align}
    &\sum_{\gamma \in \mathcal{I}_k} \Bigg( \sum_{E^{\gamma,\ell}_k \in \A^\gamma_k} \lambda^{\gamma,\ell}_k \Bigg) \left\| K\left( \1_{\bar{E}^\gamma_k}-\1_{\bar{E}^\gamma} \right)\right\|_Y  \\ &\qquad\leq C_{K, \bar E} \sum_{\gamma \in \mathcal{I}_k} \Bigg( \sum_{E^{\gamma,\ell}_k \in \A^\gamma_k} \lambda^{\gamma,\ell}_k \Bigg) \|\bar \varphi_k\|_{\bH} \leq C \sqrt{r_{\A_k}(\lambda_k)},
\end{align}
finishing the proof. Finally, we note that $ r_{\A_{k+1}}(\lambda_{k+1})=r_{\A_{k,+}}(\lambda_{k,+})$ holds by construction of $\A_{k+1}$ while $r_{\A_{k,+}}(\lambda_{k,+}) \leq r_{\A_{k,+}}(\lambda_{k,s})$ follows from $\widehat{u}^s_k= \mathcal{U}_{\A_{k,+}}(\lambda_{k,s})$.
\end{proof}
Using these results, we can finally show linear convergence of Algorithm \ref{alg:abstractonecut}. 
\begin{theorem} \label{thm:linearconvergence}
Let Assumptions \eqref{ass:maxminawayfrombdy}-\eqref{ass:quadraticgrowth} hold. Then there is $\bar{k} \geq 1$ as well as $\zeta \in (0,1)$ such that we have
\begin{equation}
 r_{\A_{k+1}}(\lambda_{k+1}) \leq \zeta r_{\A_k}(\lambda_k), \quad  r_J(u_k) \leq C_{\text{lin}} \zeta^k \quad \text{for all} \quad k \geq \bar{k}.
\end{equation}
\end{theorem}
\begin{proof}
Let $\lambda_{k,s} \geq 0$ satisfy $\widehat{u}^s_k= \mathcal{U}_{\A_{k,+}}(\lambda_{k,s})$. Proceeding similarly to Theorem \ref{thm:sublinear}, we start estimating the per-iteration descent by
\begin{equation}
\begin{aligned}
  r_{\A_{k,+}}(\lambda_{k,+})-&r_{\A_k}(\lambda_k) \leq r_{\A_{k,+}}(\lambda_{k,s})-r_{\A_k}(\lambda_k) \\
  & \leq -s\int_\Omega p_k (\widehat{v}_k-u_k) \dd x + \frac{L_{\nabla F} \|K\|s^2}{2} \|K(\widehat{v}_k-u_k)\|^2_Y \\
  & \quad + s \sum_{\gamma \in \mathcal{I}_k} \Bigg( \Bigg( \sum_{E^{\gamma,\ell}_k \in \A^\gamma_k} \lambda^{\gamma,\ell}_k \Bigg) \Per(\bar{E}^\gamma_k,\Omega) - \sum_{E^{\gamma,\ell}_k \in \A^\gamma_k} \lambda^{\gamma,\ell}_k \Per(E^{\gamma,\ell}_k,\Omega) \Bigg)
\end{aligned}
\end{equation}
where the first inequality is due to $r_{\A_{k,+}}(\lambda_{k,+}) \leq r_{\A_{k,+}}(\lambda_{k,s})$, see Lemma \ref{lem:propofwidehat}, and the second follows analogously to Theorem \ref{thm:sublinear} by Taylor expansion. We further obtain
\begin{align}
     -s\int_\Omega p_k (\widehat{v}_k-u_k) \dd x  &  + s \sum_{\gamma \in \mathcal{I}_k} \Bigg( \Bigg( \sum_{E^{\gamma,\ell}_k \in \A^\gamma_k} \!\!\lambda^{\gamma,\ell}_k \Bigg) \Per(\bar{E}^\gamma_k,\Omega) - \!\!\!\sum_{E^{\gamma,\ell}_k \in \A^\gamma_k} \!\! \lambda^{\gamma,\ell}_k \Per(E^{\gamma,\ell}_k,\Omega) \Bigg) \\
     &=s\sum_{\gamma \in \mathcal{I}_k} \Bigg( \sum_{E^{\gamma,\ell}_k \in \A^\gamma_k} \!\!\lambda^{\gamma,\ell}_k \Bigg)\bigg( -\int_{\bar{E}^\gamma_k} p_k \dd x + \Per(\bar{E}^\gamma_k,\Omega)\bigg)
     \\ & \leq s m_a \bigg( -\int_{\bar{E}_k} p_k \dd x + \Per(\bar{E}_k,\Omega)\bigg) \leq -s \frac{m_a}{M_\infty} r_{\A_k}(\lambda_k) 
\end{align}
where the equality follows from Lemma \ref{lem:propofwidehat}, the first inequality is due to Lemma \ref{lem:uniformboundslump}, noting that the summands are nonpositive and
    \begin{equation}
    -\int_{ \bar{E}_k} p_k \dd x+ \Per(\bar{E}_k,\Omega)= \sum_{j\in \mathcal{I}_k} \left\lbrack -\int_{ \bar{E}^\gamma_k} p_k \dd x+ \Per(\bar{E}^\gamma_k,\Omega) \right \rbrack 
\end{equation}
and the final one follows from Lemma \ref{lem:optifinite}. Again applying Lemma \ref{lem:propofwidehat}, we arrive at
\begin{equation}
  r_{\A_{k,+}}(\lambda_{k,+})-r_{\A_k}(\lambda_k) \leq-s \frac{m_a}{M_\infty} r_{\A_k}(\lambda_k)+s^2\frac{L_{\nabla F} \|K\| C^2_{\mathcal{D}}}{2} r_{\A_k}(\lambda_k)   
\end{equation}
for all $s\in[0,1]$. Minimizing w.r.t $s\in[0,1]$, we arrive at
\begin{align}
  r_{\A_{k+1}}(\lambda_{k+1})= r_{\A_{k,+}}(\lambda_{k,+})\leq \left( 1-\frac{m_a}{2M_\infty} \min\left\{1,\frac{m_a}{M_\infty L_{\nabla F} \|K\| C^2_{\mathcal{D}}} \right\} \right) r_{\A_k}(\lambda_k).
\end{align}
yielding $r_{\A_{k+1}}(\lambda_{k+1}) \leq \zeta r_{\A_k}(\lambda_k) $ for all $k\geq \bar{k}$ where $\zeta \in (0,1)$ is defined as above and $\bar{k}$ is chosen large enough such that all previous considerations hold. Iterating this estimate, we finally obtain
\begin{align}
    r_J(u_k) \leq r_{\A_k}(\lambda_k) \leq  \zeta^{k-\bar{k}} r_{\A_{\bar{k}}}(\lambda_{\bar{k}} )= \frac{r_{\A_{\bar{k}}}(\lambda_{\bar{k}})}{\zeta^{\bar{k}}} \zeta^k
\end{align}
finishing the proof.
\end{proof}

\section{Numerical results on triangular meshes with PDE constraints} \label{sec:numerics}
In the following, we present two numerical experiments in which we apply the presented algorithm to both elliptic and parabolic PDE-constrained control problems with distributed observations on $\Omega=(-1,1)^2$. Analogous to \cite{CriIglWal23}, we fit these into the abstract framework of \eqref{def:BVprob} by introducing a control-to-observation operator $K \colon L^q(\Omega) \to L^2(\Omega)$, mapping the control input $u$ to observations of the corresponding PDE solution $y$; see Equations \eqref{eq: K-par}, \eqref{eq: K-ell} for the precise definition of $K$ in the two settings. Considering the quadratic loss $F(\cdot)= \frac{1}{2\alpha} \|\cdot-y_d\|^2$, we arrive at
\begin{align}\tag{$\cal{DP}$}
    \min_{u \in P_0(\mathcal{T}) } \frac{1}{2\alpha} \|y-y_d\|^2_{L^2(\Omega)}+ \TV(u,\Omega) 
\end{align}
where $\alpha=10^{-4}$ is a regularization parameter and $y_d$ are given observations. As mentioned earlier, \eqref{item:assonF} and  \eqref{ass:strongconv} are satisfied for quadratic misfit functionals while \eqref{item:assonsublevel} holds if, e.g., $K$ is injective. We will briefly discuss the latter together with \eqref{item:assonK} as well as the additional regularity assumed by \eqref{ass:smoothduals} for each example individually. At this point, we also recall that the remaining assumptions, \eqref{ass:maxminawayfrombdy}, \eqref{ass:strictcomplementarity} and \eqref{ass:quadraticgrowth}, can potentially only be verified with additional effort if $\bar u$ and $\bar p$ are known analytically. This is not the case for any of the two examples considered in the following. In particular, inferring second-order properties for smooth deformations, like \eqref{ass:quadraticgrowth}, from numerical computations on triangulated sets is not possible.

 For the practical implementation of Algorithm \ref{alg:abstractonecut}, we denote by $P_0(\mathcal{T})$ and $P_1(\mathcal{T})$ the spaces of piecewise constant and piecewise linear and continuous finite elements on a pseudorandom triangulation $\mathcal{T}$ of $\Omega$. The discretized control-to-observation operator $K_h \colon P_0(\mathcal{T}) \to P_1(\mathcal{T})$ is obtained implicitly by replacing the underlying PDE solution with its finite element approximation. Moreover, as suggested in Remark \ref{rem:extension}, we keep $\Omega$ in the active set, i.e. $\Omega \in \A_k$ for all $k\in\N$. Applying Algorithm \ref{alg:abstractonecut} in the discretized setting leads to subproblems of the form  \begin{equation}\label{eq: pre-disc}
    \argmin_{E\in \S_\T(\Omega)}\; -\int_E p_k \;dx + \Per(E,\Omega)\tag{$\cal{DMC}$}
\end{equation} 
where $\S_\T(\Omega)$ denotes the class of triangulated subsets of~$\Omega$ and $p_k=-K^*_h(K_h u_k-y_d)$ can be obtained by solving one adjoint PDE, which is described in Equation \eqref{eq: K-par-ad} in the parabolic case and it corresponds to the original PDE in the elliptic setting. As described in \cite[Section 4.2.2]{CriIglWal23}, \eqref{eq: pre-disc} can be solved exactly and efficiently by reducing it to a minimal graph cut problem on the dual graph of the mesh augmented by a source $s$ and a sink $t$ and applying modern max-flow algorithms, see \cite{BoyKol04}. 
Once a new minimizer $ \bar E_k$ of \eqref{eq: pre-disc} is computed, we identify its indecomposable components $\{\bar{E}^\gamma_k\}^{n_k}_{j=1}$ by searching for strongly connected components in the residual graph. By construction, if $ \bar E_k$ corresponds to the set of nodes $I_{t}:=I\cup\{t\}$ in the dual graph, the directed edges from $I_{t}$ to its complement have zero residual capacity, ensuring that the strongly connected components within $I_{t}$ can be used to find a valid decomposition of $\bar E_k$ into its indecomposable parts. 
Practically, this computation is carried out using the NetworkX Python library, which provides an efficient implementation tailored for directed graphs. 
The worst-case complexity of this computation is $O(n+e)$, see \cite{NuSo}, where $n$ is the number of nodes, and $e$ is the number of edges, making this computation theoretically more efficient than a new cut. In practice, on a dual graph with approximately $5\cdot 10^5$ nodes, identifying the strongly connected components takes around $3$ seconds on average, compared to roughly $12$ seconds required for computing a new cut. 
Finally, we compute the observations $K _h\1_{\bar{E}^\gamma_k} $ associated with the computed components and add them to a separate list which is pruned analogously to the set $\A_k$. 
As a consequence, the solution of the finite-dimensional coefficient update problem can be realized without further PDE solves. In practice, and as already described in \cite{CriIglWal23}, this is done by employing a semismooth Newton method based on the normal map reformulation, using the weights of the previous iterate as a warm-start. Since $\A_k$ is usually small, the additional computational effort of solving these subproblems is negligible compared to the rest, leading to the solution of $n_k+1$ PDEs and the computation of one graph cut per-iteration.    
The spatial discretization for the numerical examples was performed using triangular meshes generated through the mshr component in the FEniCS framework and set up to produce a symmetric output with respect to both the $x$ and $y$ axes. These meshes contained approximately $3\cdot 10^5$ and $5\cdot 10^5$ triangles for the first and second examples, respectively. In both settings, the algorithm is run until the convergence indicator
$$j_k:=\int_{\bar{E}_k} p_{k} \dd x - \Per(\bar{E}_k, \Omega)\geq 0,$$
is smaller than $10^{-10}$. In view of Lemma \ref{lem:optifinite}, $j_k$ is, up to a multiplicative constant given by the $L^\infty$ norm of the sought solution, an upper bound on the residual $r_J(u_k)$.

All computations were carried out on a 2021 MacBook Pro featuring a 10-core M1 Max CPU. The Python code for our implementation and configuration details to reproduce the examples presented can be found at \url{https://doi.org/10.5281/zenodo.15231157}.

\subsection{A parabolic example}
\begin{figure}[ht]
    \centering
    \subfigure[Toy control $u_d$]{\includegraphics[width=0.45\textwidth]{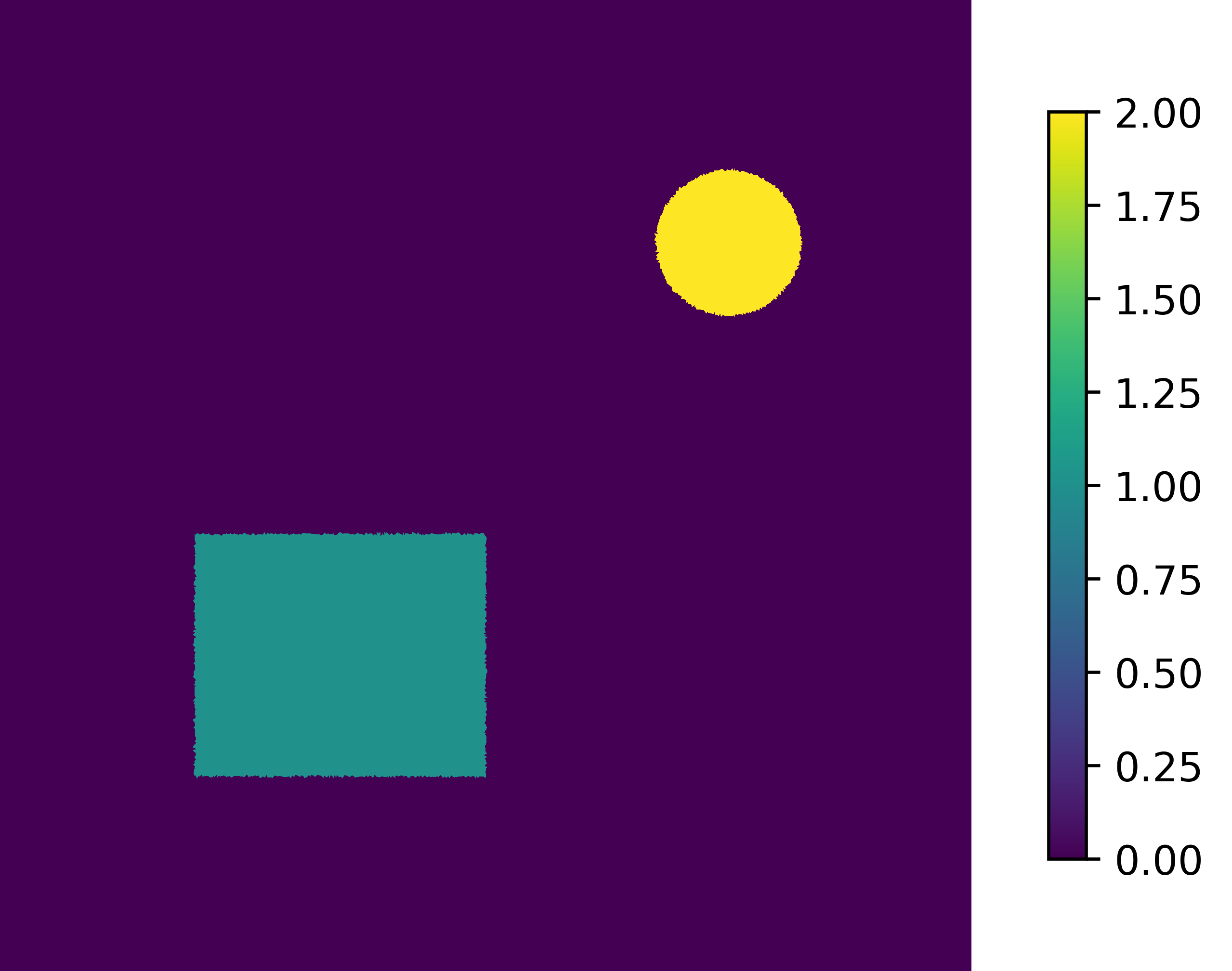}\label{fig:toy}}\hfill
    \subfigure[Computed minimizer $\bar{u}$]{\includegraphics[width=0.45\textwidth]{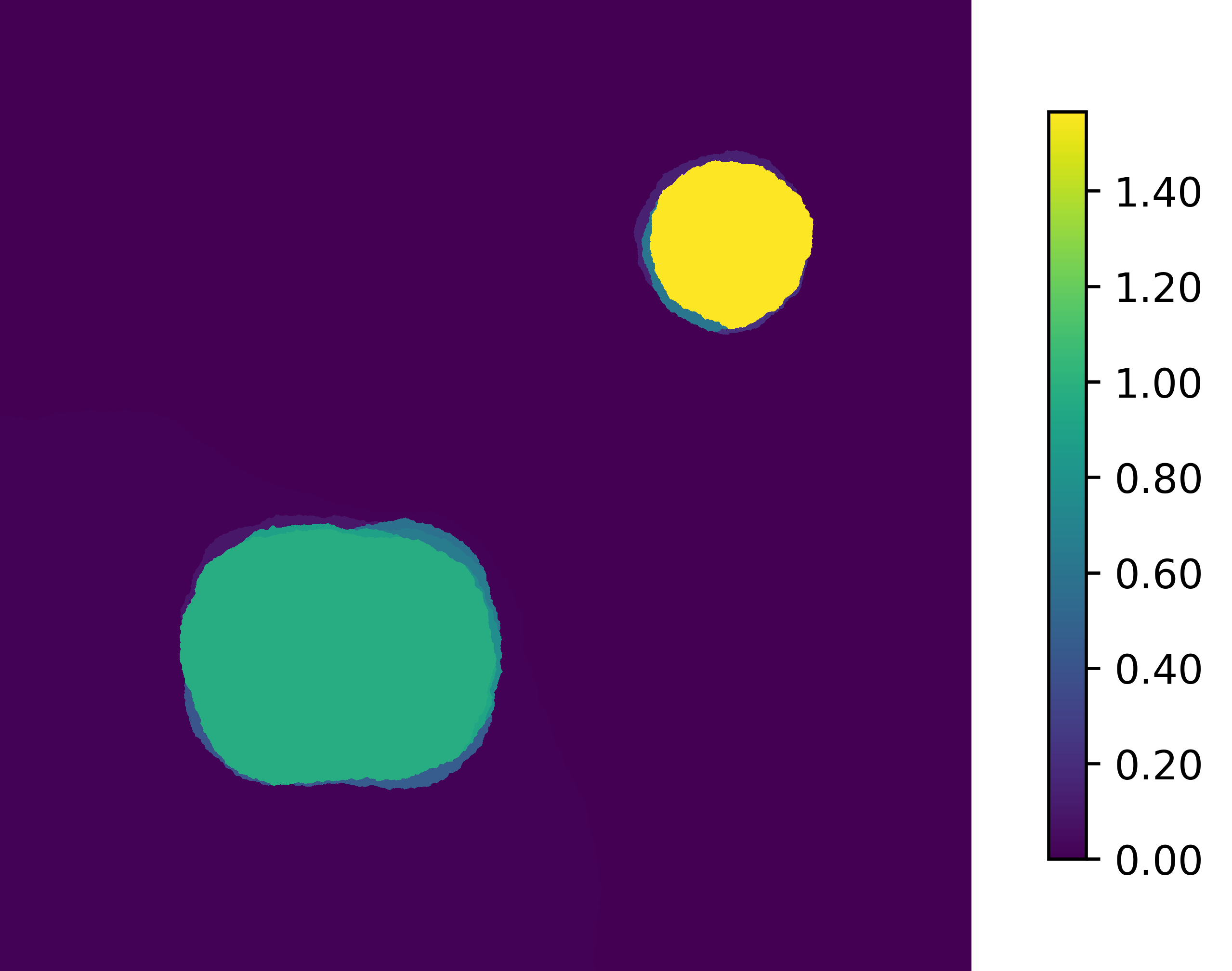}\label{fig:solution}}
    \caption{Toy control $u_d$ and output $\bar{u}$ of Algorithm \ref{alg:abstractonecut} for the parabolic problem}
    \label{fig:parabolic}
\end{figure}
First, we consider  a classical inverse problem inspired by inverse heat conduction, the parabolic problem
\begin{equation}\label{eq: K-par}
    \partial_t y-\Delta y + \frac{1}{2}y=0~ \text{in} \; (0,T)\times \Omega,\quad 
            y=0~ \text{in} \; [0,T]\times \partial \Omega, \quad 
            y=u  ~\text{in} \; \{0\}\times \Omega,
\end{equation}
in which the control $u$ enters as initial condition and endtime observations at $T=0.02$ are considered, that is to say $Ku=y(T,\cdot) \in L^2(\Omega)$. The corresponding adjoint operator satisfies $K^*z=q(0,\cdot)$ where $q$ is determined from 
\begin{equation}\label{eq: K-par-ad}
    -\partial_t q-\Delta q + \frac{1}{2}q=0~ \text{in} \; (0,T)\times \Omega,\quad 
            q=0~ \text{in} \; [0,T]\times \partial \Omega, \quad 
            q=z  ~\text{in} \; \{T\}\times \Omega,
\end{equation}
Note that $K$ is well-defined and continuous in the sense of \eqref{item:assonK}, cf. \cite[Lemma 2.3]{Casas15}, while its injectivity follows by a backwards uniqueness argument as in \cite[Theorem 2.4]{Casas15}. Given $z \in L^2(\Omega)$, we further obtain $p>2$ with 
\begin{align}
\|K^* z\|_{W^{2,p}(\Omega)}  \leq C_1  \|(-\Delta q+(1/2)\operatorname{Id})\lbrack K^* z\rbrack\|_{H^2(\Omega)} \leq C_2 \|z\|_{L^2(\Omega)}.  
\end{align}
Here, the right estimate follows from a semigroup argument, see \cite[Remark 2.2, Lemma 2.5]{Vexler20}, while the left one is a standard regularity result on convex, polygonal domains. Together with the Sobolev embedding $W^{2,p}(\Omega) \hookrightarrow^c C^1(\cl \Omega)$, we conclude that \eqref{ass:smoothduals} holds.

For time integration, we apply the implicit Euler scheme with a uniform partition into $9$ subintervals and set $y_d=Ku_d$ where $u_d$ is the piecewise constant function depicted in Figure \ref{fig:toy}.

The reconstructed minimizer obtained from Algorithm \ref{alg:abstractonecut} in the planar (parabolic) setting is shown in Figure \ref{fig:solution}. As expected, the computed approximation closely resembles $u_d$ while exhibiting a loss of contrast and noticeable smoothing of the geometry of level sets due to the nearly isotropic TV-regularization afforded by the use of a pseudo-random mesh. These are both features which are expected from the original formulation \eqref{def:BVprob} already in the continuum setting, so they are not related to the choice of discretization or algorithm. However, we also observe the discrete reconstruction features four nonzero level sets clustered around the two present in the toy control. Given the smoothness of the dual variable, we would not expect such additional level sets in the continuous solution. Hence, we attribute these artifacts to the discretization of the model and the greedy nature of the algorithm.

\subsection{An elliptic example}\label{sec:elliptic}
As a second example, we consider an elliptic problem which, although not physically relevant, is used as a standard test case for total variation regularized optimal control. It originated in \cite{ClaKun11} and was also used in \cite{HafMan22} and \cite[Section 6.3]{CriIglWal23}. More in detail, we set $y_d= \1_{(-0.5,0.5)^2}$ and $Ku=y$ where $y$ satisfies
\begin{equation}\label{eq: K-ell}
-\Delta y= u~ \text{in } \Omega, \quad y=0 ~ \text{on } \partial\Omega.
\end{equation}
In this case, we have $K^* z= Kz$ for $z \in L^2(\Omega)$, and the continuity of $K$, as well as the injectivity of $K$ follow immediately by Lax-Milgram plus $H^1(\Omega) \hookrightarrow^c L^2(\Omega)$; see \eqref{item:assonK}. However, we point out that $Kz \not \in C^1(\cl \Omega) $ for $z \in L^2(\Omega) $, i.e. \eqref{ass:smoothduals} is not satisfied.

In order to be comparable to the previous results in \cite{CriIglWal23}, we drop the nonnegativity constraints on $u$ and augment Algorithm \ref{alg:abstractonecut} according to Remark \ref{rem:extension}. The computed minimizer together with $y_d$ is depicted in Figure \ref{fig:spheres}. Note that $\bar{u}$ exhibits more complex structural features due to $-\Delta y_d \not \in \BV(\Omega)$. In particular, we point out that disjoint components of level sets of $\bar{u}$ have intersecting boundaries, and jumps occur on the boundary of the domain, suggesting that \eqref{ass:maxminawayfrombdy} does not hold in the continuous setting.

\begin{figure}[ht]
    \centering
    \subfigure[Desired state $y_d$]{\includegraphics[width=0.43\textwidth]{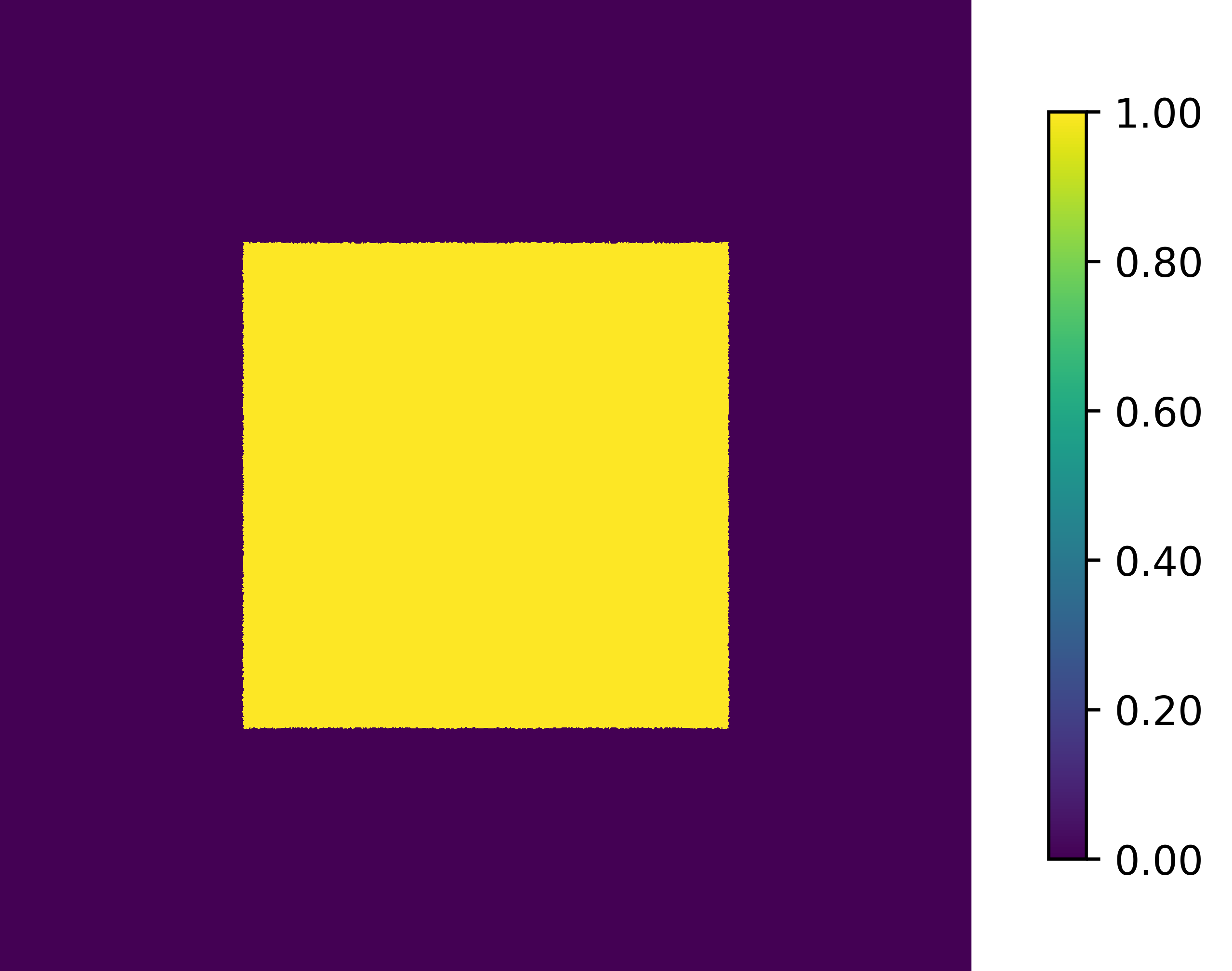}\label{fig:input}}\hfill
    \subfigure[Computed minimizer $\bar{u}$]{\includegraphics[width=0.45\textwidth]{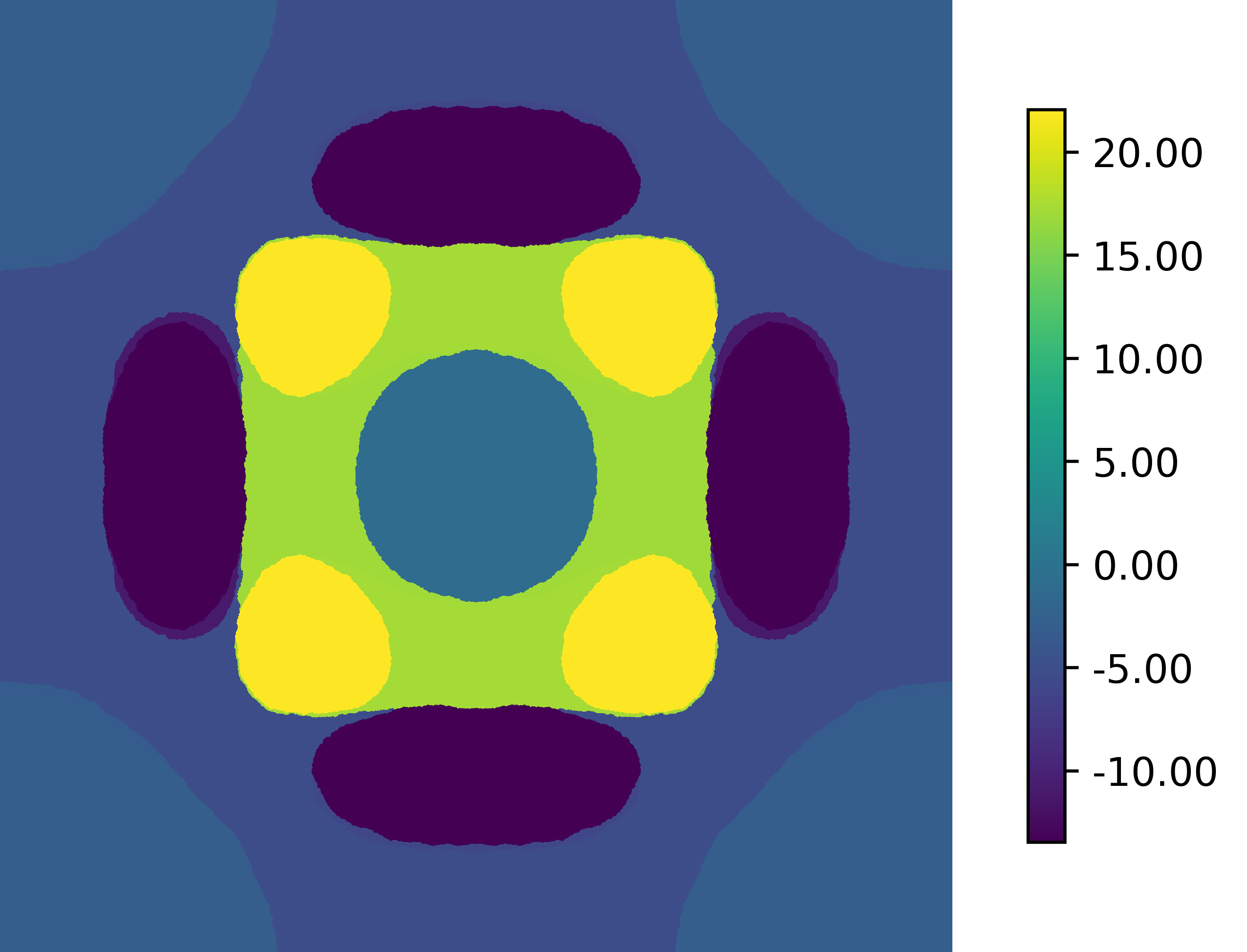}\label{fig:castle}}
    \caption{Desired state $y_d$ and output $\bar{u}$ of Algorithm \ref{alg:abstractonecut} for the elliptic problem}
    \label{fig:spheres}
\end{figure}

\subsection{Practical performance and discussion}
\begin{figure}[b!]
    \centering
    \subfigure[Convergence indicator $j_k$ and residual $r_J(u_k)$ for the parabolic example.]{\includegraphics[width=0.48\textwidth]{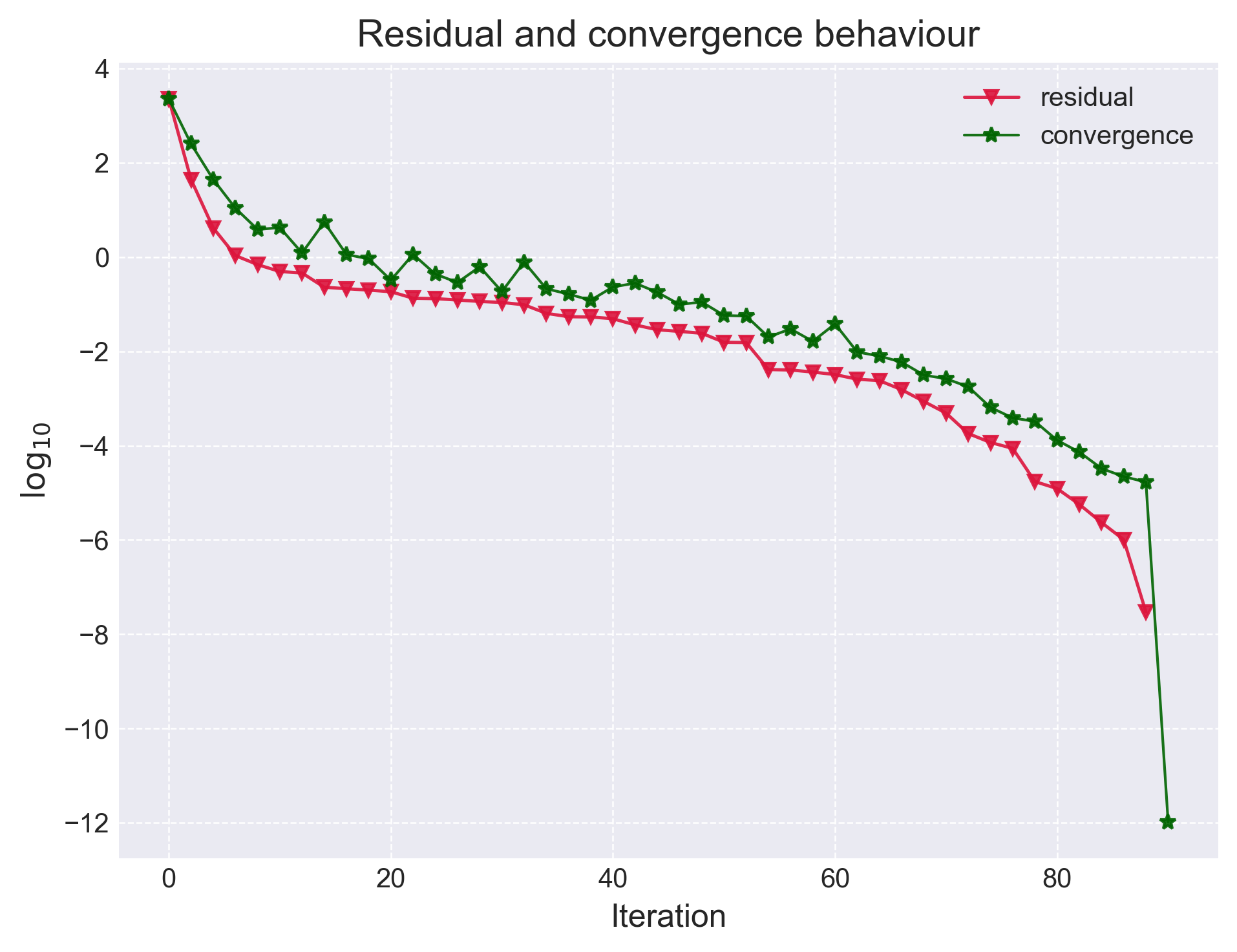}\label{fig:conv-res}}\hfill
    \subfigure[Residual $r_J(u_k)$ in the elliptic example compared with \cite{CriIglWal23} up to the second to last iteration.]{\includegraphics[width=0.48\textwidth]{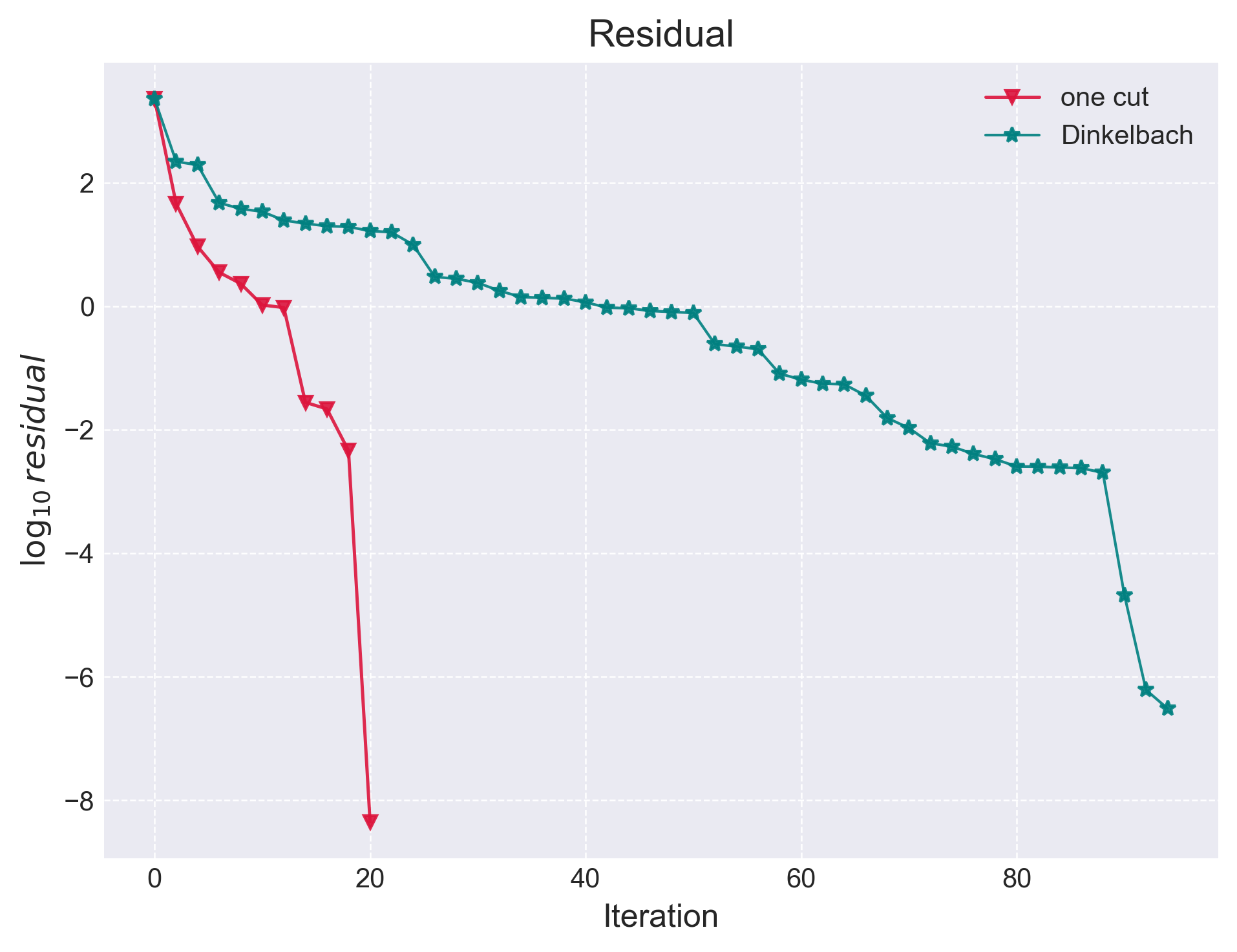}\label{fig:comp}}
    \caption{Convergence rate and residual of Algorithm \ref{alg:abstractonecut} in the two examples.}
    \label{fig:parabolic-rates}
\end{figure}
 We briefly discuss the practical performance of Algorithm \ref{alg:abstractonecut} from a quantitative perspective. For the parabolic problem, we plot the evolution of the convergence indicator $j_k$ as well as of the residual $r_J(u_k)\approx J(u_k)-J(\bar u)$ in Figure \ref{fig:conv-res}. For the latter, we observe a, at least, linear rate of convergence while the former vanishes abruptly in the last iteration which can be attributed to finite-step convergence on the discrete level as in \cite{CriIglWal23}. For the elliptic example, Algorithm \ref{alg:abstractonecut} is compared to the method presented in \cite{CriIglWal23}. The evolution of both residuals is plotted in Figure \ref{fig:comp}. While both algorithms exhibit a considerably faster than sublinear convergence behavior even in this example which in the continuum would not be covered in the setting of Section \ref{subsec:linear}, Algorithm \ref{alg:abstractonecut} remarkably outperforms the original method in terms of iterations. In the present example, we believe that this is due to the additional splitting of $\bar{E}_k$ into indecomposable components which allows for greater flexibility in the update step for the iterate and, implicitly, exploits the symmetry of the minimizer. In order to compare the numerical effort of both methods, we recall that our implementation of Algorithm \ref{alg:abstractonecut} requires the resolution of $n_k+1$ PDEs and the computation of one graph cut per-iteration while the method in \cite{CriIglWal23} requires $2$ PDE solves and several graph cuts. 
 In the present example, this leads to a combined amount of $72$ PDE solves and $21$ graph cuts for Algorithm \ref{alg:abstractonecut} while its counterpart requires significantly more, namely $192$ PDE solves and $426$ graph cuts. This observation is also reflected in a vastly decreased computation time, with the previous method taking 3 hours in comparison to 26 minutes for the new method. Finally, we point out that all of our convergence results where derived on the infinite-dimensional function space level while the practical realization relies on discretized, finite-dimensional instances. Following such an optimize-then-discretize philosophy, one can often hope for a mesh-independent convergence behavior, i.e. the number of iterations to reach the convergence criterion is observed or proven to be independent of the underlying spatial resolution. We experimentally verify this for both problems in Figure \ref{fig:placeholder} and note that, for both problems, the algorithm requires less steps on finer grids which was previously already observed for conditional gradient related algorithms, see \cite{sachs}.
 \begin{figure}
     \centering
     \subfigure[Residual $r_J(u_k)$ for the elliptic example]{\includegraphics[width=0.48\linewidth]{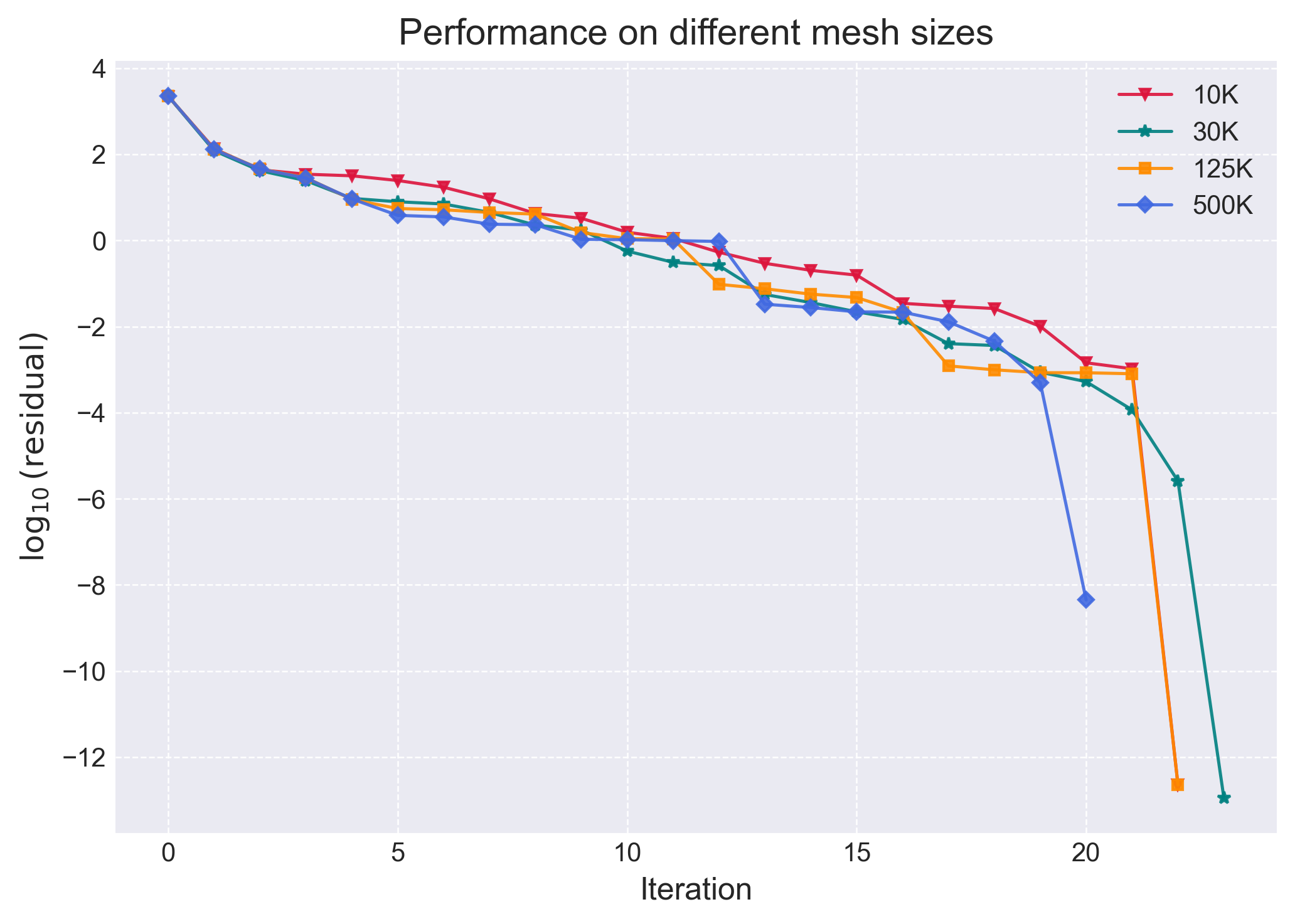}}
     \subfigure[Residual $r_J(u_k)$ for the parabolic example]{\includegraphics[width=0.48\linewidth]{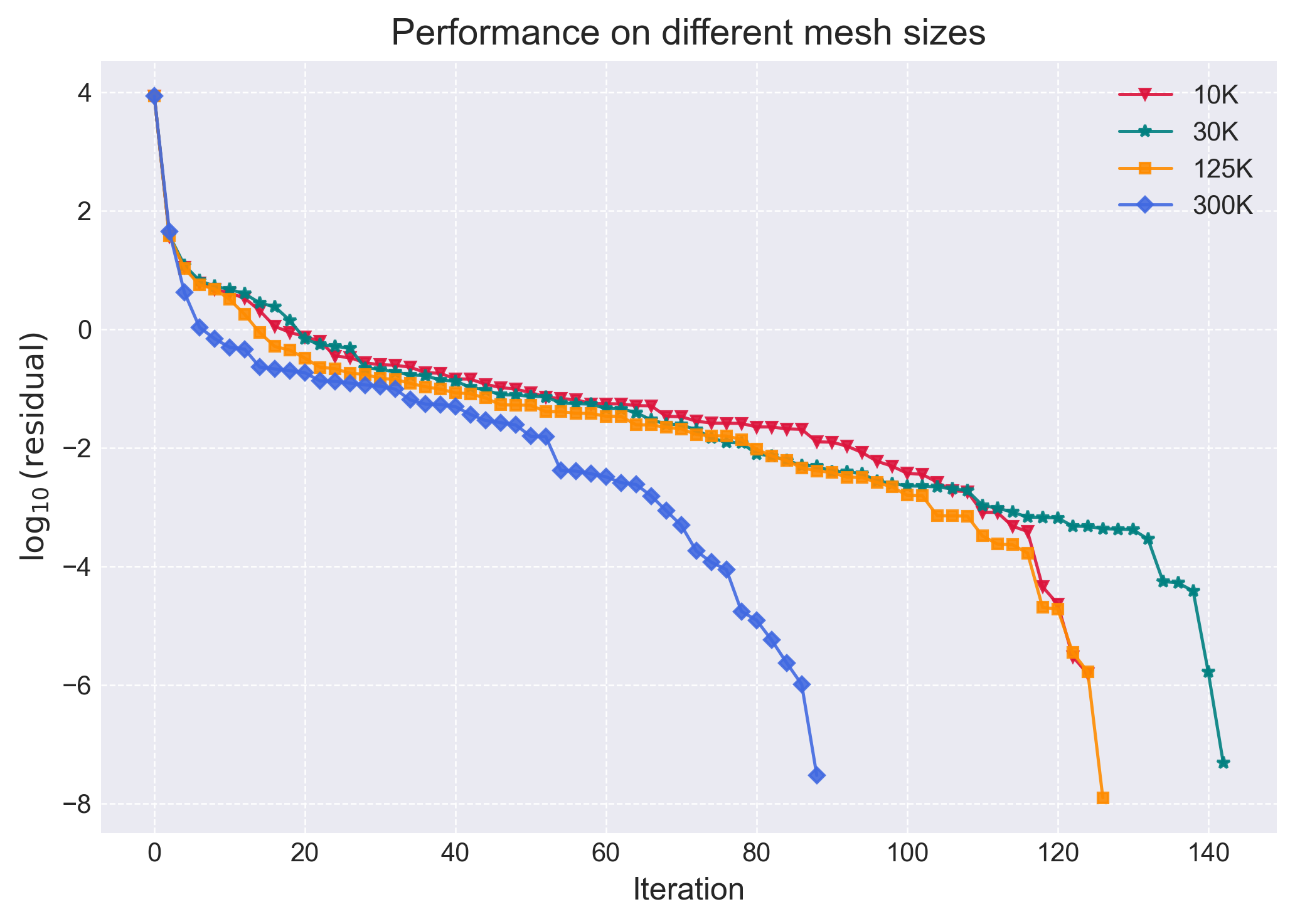}}
     \caption{Residual of Algorithm \ref{alg:abstractonecut} with different numbers of degrees of freedom.}
     \label{fig:placeholder}
 \end{figure}

\appendix
\section{Extending the prescribed curvature problem from \texorpdfstring{$\Omega$}{Omega} to \texorpdfstring{$\R^d$}{Rd}}\label{sec:appendix}

\begin{proposition}\label{prop:existenceMC}Let $p \in L^d(\R^d)$ with compact support, and $\Omega \subset \R^d$ be a bounded domain. Then the problems
\[\argmin_{E \subset \R^d} - \int_E p \dd x + \Per(E, \R^d) \quad \text{and} \quad \argmin_{E \subset \Omega} - \int_E p \dd x + \Per(E, \Omega)\]
both admit at least one minimizer. If additionally $d \leq 7$ and there is a $q>d$ such that $p \in L^q(\R^d)$ (respectively $p \in L^q(\Omega)$), then minimizers $E$ satisfy that $\partial E$ (respectively $\partial E \cap \Omega$) are $C^{1,\beta}$ hypersurfaces for some $\beta \in (0,1)$. Finally, if additionally $d=2$ and $p \in \cC^1(\R^2)$, for minimizers $E$ of the unconstrained problem we have that $\partial E$ is of class $\cC^{3}$.
\end{proposition}
\begin{proof}
Existence of minimizers is obtained by the direct method of the calculus of variations with respect to $L^1$ convergence of sets, that is, with $E_n$ converging to $E$ whenever $|E_n \Delta E|$ converges to zero. The perimeters are lower semicontinuous with respect to this convergence \cite[Prop.~3.38(c)]{AmbFusPal00}, and the integral terms are continuous by the integrability assumption on $p$. Coercivity is ensured by standard compactness results for sets with perimeter bounds \cite[Thm.~3.39]{AmbFusPal00}, which require them to be enclosed in a commmon set of finite Lebesgue measure. This is ensured by the boundedness of $\supp p$ and $\Omega$, respectively.

The $C^{1,\beta}$ regularity results can be found in \cite[Thm.~3.1]{Mas75} for all $q>d$, or \cite[Thm.~21.8]{Mag12} for the special case $q=\infty$. The higher regularity claim is part of the proof of \cite[Prop.~4.1]{DecDuvPet24}.
\end{proof}

\begin{proposition}\label{prop:globalize}
Assume that $\Omega \subset \R^d$ is bounded, open and convex, $m \geq 0$, $\bar{p} \in \cC^m(\cl \Omega)$, and that the maximal minimizer $\bar{E}$ of the prescribed mean curvature problem in $\Omega$ with curvature $\bar{p}$ (left hand side of \eqref{eq:farfrombdybar} below) satisfies $\dist(\bar{E}, \partial \Omega)>0$. Then there exists a $\cC^m$ extension $\widehat{p}$ of $\bar{p}$ to $\R^d$ such that
\begin{equation}\label{eq:farfrombdybar}
\argmin_{E \subset \Omega} - \int_E \bar{p} \dd x + \Per(E, \Omega) = \argmin_{E \subset \R^d} - \int_E \widehat{p} \dd x + \Per(E,\R^d).
\end{equation}
Further, if $p_k \to \bar{p}$ strongly in $L^d(\Omega)$, then for all $k$ large enough the corresponding maximal minimizers $\bar{E}_k$ also satisfy $\dist(\bar{E}_k, \partial \Omega)>0$, and there exist smooth extensions $\widehat{p}_k$ of $p_k$ to $\R^d$ such that
\begin{equation}\label{eq:farfrombdyk}
\argmin_{E \subset \Omega} - \int_E p_k \dd x + \Per(E, \Omega) = \argmin_{E \subset \R^d} - \int_E \widehat{p}_k \dd x + \Per(E,\R^d).
\end{equation}
If additionally $p_k \to \bar{p}$ in $\cC^m(\Omega)$, these extensions can be chosen such that 
\begin{equation}\label{eq:contdepext}
\big\|\widehat{p}_k - \widehat{p}\big\|_{\cC^m(\R^d)} \xrightarrow[k \to \infty]{} 0.
\end{equation}
\end{proposition}
\begin{proof}Let us start with the inclusion of the right-hand side into the left-hand side of \eqref{eq:farfrombdybar}. For any $\delta>0$, we can find a $\cC^m$ extension $\bar{p}_\delta \in \cC^m_c(\R^d)$ with $\bar{p}_\delta = \bar{p}$ on $\Omega$ and $\supp \bar{p}_{\delta} \subset \{x \in \R^d \,\vert\, \dist(x,\Omega) < \delta\}$, which can be constructed for example using partitions of unity as in \cite[Lem.~2.26]{Lee13}. With these, consider the family of problems
\begin{equation}\label{eq:prescurvdelta}\min_{E \subset \R^d} - \int_E \bar{p}_\delta \dd x + \Per(E,\R^d).\end{equation}
We claim that for some $\delta_0$ small enough, minimizers of these problems with $\delta \leq \delta_0$ are also minimizers for the interior problem 
\begin{equation}\label{eq:prescurvbar}\min_{E \subset \Omega} - \int_E \bar{p} \dd x + \Per(E,\Omega),\end{equation}
of which $\bar{E}$ is the maximal minimizer. Assume this was not the case, meaning that there exists a sequence $\delta_n \to 0$ and minimizers $E_{\delta_n}$ with 
\begin{equation}\label{eq:deltasubopt}- \int_{E_{\delta_n}} \bar{p} \dd x + \Per(E_{\delta_n},\Omega) > - \int_{\bar{E}} \bar{p} \dd x + \Per(\bar{E},\Omega) \quad \text{for all }n.\end{equation}
Since $\supp \bar{p}_{\delta_n} \subset \{x \in \R^d \,\vert\, \dist(x,\Omega) < \delta_n\}$ , the latter is convex and intersection with convex sets cannot increase $\Per(\cdot,\R^d)$ \cite[Ex.~15.14]{Mag12}, we may assume that 
\begin{equation}\label{eq:distanceleqdelta}\left| \{x \in \R^d \,\vert\, \dist(x,\Omega) < \delta_n\} \setminus E_{\delta_n} \right|=0,\end{equation} which in particular implies that the sequence $|E_{\delta_n}|$ is bounded. 

We also have that $\Per(E_{\delta_n},\R^d) \leq \int_{\R^d} \bar{p}_{\delta_n} \dd x$ is bounded, so by the same compactness results used in the proof of Proposition \ref{prop:existenceMC}, for a not relabelled subsequence the $E_{\delta_n}$ converge in $L^1$ to some $E_0 \subset \R^d$ with $\1_{E_{\delta_n}} \wksto \1_{E_0}$ in $\BV(\R^d)$, and in fact
\begin{equation}\label{eq:E0min}E_0 \in \argmin_{E \subset \R^d} - \int_E \mathring{p} \dd x + \Per(E,\R^d),\end{equation}
where $\mathring{p}$ is the extension by zero of $\bar{p}$. Moreover, by \eqref{eq:distanceleqdelta} also $\left|E_0 \setminus \cl \Omega\right| = 0$, and in fact $E_0 \subset \cl \Omega$. If this would not be the case, there is $x_0 \in \operatorname{int}(E_0) \setminus \cl \Omega$ or $x_0 \in \partial E_0 \setminus \Omega$. In both cases, we obtain the existence of $r_0>0$ such that $B(x_0,r_0) \subset \R^d \setminus \cl \Omega$ and $|E_0 \cap B(x_0,r_0)| \neq 0$ which follows by definition of the interior in the first case and by \eqref{eq:bdy} in the second, leading to a contradiction. This inclusion implies that $E_0$ is admissible for \eqref{eq:prescurvbar}, that is
\begin{equation}\label{eq:Ebarmin}- \int_{E_0} \bar{p} \dd x + \Per(E_0, \Omega) \geq - \int_{\bar{E}} \bar{p} \dd x + \Per(\bar{E},\Omega).\end{equation}
Moreover, $\bar{E}$ is also admissible in \eqref{eq:E0min} and $\Per(\bar{E},\Omega) = \Per(\bar{E},\R^d)$ by assumption, meaning that all of these energies must be equal. Now, we claim that there holds
\begin{equation}\label{eq:barEbarrier}
  \dist(E_0,\partial \Omega) \geq \dist(\bar{E}, \partial \Omega)
\end{equation}
where we again recall that this is understood in the sense of representatives satisfying \eqref{eq:bdy}. For this we assume the opposite, which since $\bar{E} \subset \Omega$ by definition, implies the existence of $x \in \partial E_0$ as well as $\varepsilon>0$ such that
\begin{equation}
    \dist(x, \partial \Omega) < \dist(y, \partial \Omega)-\varepsilon \quad \text{ for all }y \in \partial \overline{E},
\end{equation}
which in combination with \eqref{eq:bdy} would imply $|E_0\setminus \bar{E}| >0$, contradicting maximality of $\bar{E}$ among minimizers of the left hand side of \eqref{eq:farfrombdybar}.

Next, we want to show that for all $n$ large enough, $E_{\delta_n}\subset \Omega$ and $\dist(E_{\delta_n},\partial \Omega) >0$. Take any $n$ for which this is not the case, which implies that there is some $x_{\delta_n} \in \partial E_{\delta_n}$ with either $x_{\delta_n} \notin \Omega$ or $x_{\delta_n} \in \partial \Omega$. But we notice that the $E_{\delta_n}$ possess uniform density estimates, meaning that there are $c \in (0,1)$ and $r_0>0$ such that for all $0<r\leq r_0$, all $n$ and all $x \in \partial E_{\delta_n}$, we have
\[1-c \leq \frac{|E_{\delta_n}\cap B(x,r)|}{|B(x,r)|} \leq c.\]
These directly imply that
\[\left|E_{\delta_n}\cap B\big(x_{\delta_n},r_1\big)\right| \geq (1-c)\left|B\big(x_{\delta_n},r_1\big)\right| \quad \text{ for }r_1 := \min(r_0, \dist(\bar{E}, \partial \Omega)).\]
But if $n$ is large enough so that 
\begin{equation}\label{eq:vanishingnuisance}\left|E_{\delta_n} \setminus \{x \in \Omega \,\vert\, \dist(x,\partial\Omega) \geq \dist(\bar{E}, \partial \Omega)\}\right| \leq |E_{\delta_n}\Delta E_0| < (1-c)|B(0,r_1)|,\end{equation}
we immediately get a contradiction (note that we have used \eqref{eq:barEbarrier} for the first inequality). We obtain that there is $n_0$ such that if $n \geq n_0$ then we have $E_{\delta_n}\subset \Omega$ and $\dist(E_{\delta_n},\partial \Omega) >0$. This then implies that
\[- \int_{E_{\delta_{n}}} \bar{p}_{\delta_n} \dd x + \Per(E_{\delta_n},\R^d) = - \int_{E_{\delta_n}} \bar{p} \dd x + \Per(E_{\delta_n},\Omega) \quad\text{for all }n\geq n_0,\]
but in combination with \eqref{eq:deltasubopt} this means that $E_{\delta_{n}}$ was not a minimizer of \eqref{eq:prescurvdelta}, since $\bar{E}$ is also admissible for it. This contradiction implies that we can use $\widehat{p}=\bar{p}_{\delta_0}$ as the desired extension of $\bar{p}$. Here, we notice that $\delta_0$ in principle depends not only on $r_1$ but also the rate of convergence of the left-hand side of \eqref{eq:vanishingnuisance} to zero. Finally, the opposite inclusion in \eqref{eq:farfrombdybar} follows immediately, since we have proved that for minimizers $E$ of the unconstrained problem, in fact $\Per(E,\Omega) = \Per(E,\R^d)$.

Next, we want to show that for $k$ large enough and some $\delta_{0,k}$, the $p_{k,\delta_{0,k}}$ constructed as above can be used as the desired extension $\widehat{p}_k$ of $p_k$. For this, we consider the maximal minimizer $\bar{E}_k$ of 
\begin{equation}\label{eq:prescurvbark}\min_{E \subset \Omega} - \int_E p_k \dd x + \Per(E,\Omega).\end{equation}
Moreover, arguing by compactness as above, $\bar{E}_k$ converge to some $\check{E}$ in $L^1$ and weak-* of indicator functions, and $\check{E}$ is a minimizer of \eqref{eq:prescurvbar}, so $\check{E} \subset \bar{E}$ by maximality of the latter. Since $p_k \to \bar{p}$ strongly in $L^d$, we can obtain density estimates for all minimizers of \eqref{eq:prescurvbark} in which the corresponding $c \in (0,1)$ and $r_0>0$ are independent of $k$. In combination with $|\bar{E}_k \Delta \check{E}| \to 0$ implies that $d_H(\bar{E}_k,\check{E}) \to 0$, where $d_H$ is the Hausdorff distance, and in particular 
\[\dist(\bar{E}_k, \partial \Omega)>\frac12 \dist(\bar{E}, \partial \Omega) \quad \text{for }k \geq k_0.\]
Thus, for such $k$ we can obtain a corresponding $\delta_{0,k}$ such that \eqref{eq:farfrombdyk} holds for $\widehat{p}_k=p_{k,\delta}$ for all $\delta \leq \delta_{0,k}$.

It remains to prove \eqref{eq:contdepext}. The main obstacle is that in the proof given the choice of $\delta$ for the extension depends on the rate of convergence of the maximal minimizers $\bar{E}_{k,\delta}$ of the unconstrained problem with $p_{k,\delta}$, which could prevent a choice of $\delta$ independent of $k$. For this, we can define for each $\delta$ a new function $p_{M,\delta}:\Omega \to \R$ by
\begin{equation}
p_{M,\delta}(x):=\max\Big(\bar{p}_\delta(x), \sup_{k \geq k_0} p_{k,\delta}(x)\Big)
\end{equation}
and the corresponding maximal minimizers $\bar{E}_{M,\delta}$, for which we can use a comparison principle for sets of prescribed mean curvature (see for example \cite[Lem.~3.4]{IglMer21}) to obtain that $\bar{E}_{k,\delta} \subset \bar{E}_{M,\delta}$, but also by the restriction $k \geq k_0$ that
\[\liminf_{\delta\to 0}\dist\Big( E_{M,\delta}, \partial \Omega\Big)>0.\]
Finally, we conclude by noticing that for fixed $\delta >0$ we have a continuous dependence 
\[\|p_\delta-q_\delta\|_{\cC^m(\R^d)} \leq \omega\left(\|p-q\|_{\cC^m(\cl \Omega)}\right)\] 
for some modulus of continuity $\omega$. This continuity is not obvious from the standard construction of smooth extensions but can be obtained for example using the extensions to $\R^d$ provided by \cite[Thm.~1]{Fef07} and multiplying them by a bump function which is identical to $1$ on $\cl \Omega$ and supported on $\{x \in \R^d \,\vert\, \dist(x,\Omega) < \delta\}$.
\end{proof}
\section*{Acknowledgments}
The authors would like to thank two anonymous reviewers of the paper whose comments significantly improved its contents.
\small
\bibliographystyle{plain}
\bibliography{onecut}
\end{document}